\documentclass[11pt]{amsart}
\usepackage{graphicx}
\usepackage{url}
\usepackage{amsfonts}
\usepackage{amscd}
\usepackage{amssymb}
\usepackage{alltt}
\usepackage[T1,T5]{fontenc}

\def\op#1{{\text{#1}}}

\newcommand{\ring}[1]{\mathbb{#1}}
\def\lto{\ensuremath{\,\leadsto\,}}
\def\dih{\op{dih}}
\def\line{$\ell$}
\def\text{\hbox}
\def\sz{small} 
\def\|{{\hskip0.1em|\hskip-0.15em|\hskip0.1em}}

\def\coloneq{\mathrel{\mathop:}=}
\def\sfrac#1#2{{\textstyle \frac {#1} {#2}}}
\def\eqref#1{(\ref{#1})}
\def\rmx{\rm}
\def\vn{\fontencoding{T5}\selectfont}
\def\calc#1{{\textsc{calc-#1}}}

\def\pt{\mathrm{pt}}
\def\doct{\delta_{\mathrm{oct}}}
\def\asolid{\mathrm{a}}
\def\sqroot{\mathrm{sqrt}}
\def\rcp{\mathrm{rcp}}
\def\bstein{\mathrm{B}}

\newcommand{\real}{\mbox{$\protect\mathbb R$}}

\let\And=\wedge                    

\newcommand{\Imp}{\Rightarrow}
\newcommand{\Iff}{\Leftrightarrow}

\newcommand{\BEQ}{\mbox{\raise4pt\hbox{$\ulcorner$}}}

\newcommand{\EEQ}{\mbox{\raise4pt\hbox{$\urcorner$}}}

\let\subset=\subseteq

\newcommand{\BA}{\begin{array}[t]{l}}
\newcommand{\EA}{\end{array}}

\makeatletter
\def\imod#1{\allowbreak\mkern10mu({\operator@font mod}\,\,#1)}
\makeatother

\newlength{\hsbw}
\newtheorem{thm}{Theorem}
\newtheorem{lemma}{Lemma}

\newtheorem{remark}{Remark}

\parskip=\baselineskip

\begin{document}

\title{A revision of the proof of the Kepler conjecture}

\author[Hales]{Thomas C. Hales}
\address[T.~Hales]{Math Department, University of Pittsburgh}
\email{hales@pitt.edu}

\author[Harrison]{John Harrison}
\address[J.~Harrison]{Intel Corporation,  JF1-13, 2111 NE 25th Avenue
 Hillsboro, OR 97124
 USA}
\email{johnh@ichips.intel.com}

\author[McLaughlin]{Sean McLaughlin}
\address[S.~McLaughlin]{Carnegie Mellon University}
\email{seanmcl@gmail.com}

\author[Nipkow]{Tobias Nipkow}
\address[T.~Nipkow]{Department for Informatics, Technische Universit\"at
M\"unchen}
\email{\url{www.in.tum.de/~nipkow}}

\author[Obua]{Steven Obua}
\address[S.~Obua]{Department for Informatics, Technische Universit\"at
M\"unchen}

\author[Zumkeller]{Roland Zumkeller}
\address[R.~Zumkeller]{\'Ecole Polytechnique, Paris}

\begin{abstract}
The Kepler conjecture asserts that no packing of congruent balls in three-dimensional Euclidean space 
has density greater than that of the face-centered cubic packing.  The original proof,
announced in 1998 and published in 2006, is long and complex. The process of  revision and review did not end with the publication of the proof.
This article summarizes the current status of a long-term initiative to
reorganize the original proof into a more transparent form and to provide a greater
level of certification of the correctness of the computer code and other details of the proof.
A final part of this article lists errata in the original proof of the Kepler conjecture.
\end{abstract}

\maketitle
\setcounter{tocdepth}{1}

{
\parskip=0pt

\tableofcontents

}

\section*{Introduction}

In 2006, {\it Discrete and Computational Geometry} devoted an issue to the
proof of the Kepler conjecture on sphere packings, which asserts that no
packing of congruent balls in three-dimensional Euclidean space can
have density greater than that of the face-centered cubic packing \cite{Hales:2006:DCG}, \cite{Hales:2006:DCG:3}, \cite{Hales:2006:DCG:4}, \cite{Hales:2006:DCG:6}, \cite{Ferguson:2006:DCG:5}.

The proof is long and complex.  The editors' forward to that issue remarks
that ``the reviewing of these papers was a particularly enormous and daunting task.''
``The main portion of the reviewing took place in a seminar run at E\"otvos University
over a 3 year period.  Some computer experiments were done in a detailed check.
The nature of this proof, consisting in part of a large number of inequalities having
little internal structure, and a complicated proof tree, makes it hard for humans
to check every step reliably.  Detailed checking of specific assertions found them to
be essentially correct in every case tested.  The reviewing process produced in the
reviewers a strong degree of conviction of the essential correctness of the proof
approach, and that the reduction method led to nonlinear programming problems of
tractable size.''

The process of review and revision did not end when the proof was published.
This article summarizes the current status of a long-term initiative to
reorganize the original proof into a more transparent form and to provide a greater
level of certification of the correctness of the computer code and other details of the proof.  

The article contains two parts.  The first part describes
an initiative to give a formal proof of the Kepler conjecture.  The second part gives errata in the original proof of the Kepler conjecture.  Most of these errata are minor.  However, a significant new argument appears in a separate section (Section~\ref{sec:biconnected}).  It finishes an incomplete argument in the original proof asserting that there is no loss in generality in assuming (for purposes of the main estimate) that subregions are simple polygons.  The incomplete argument was detected during the preparation of the blueprint edition of the proof, which is described in Section~\ref{sec:blueprint}.

In this article, the {\it original proof} refers to the proof published
in \cite{Hales:2006:DCG}.
S. Ferguson and T. Hales take full responsibility for every possible error in the original proof of the Kepler conjecture.  Over the past decade, many have contributed significantly to making that proof more reliable.

\part{Formal Proof Initiative}

\section{The Flyspeck project}

The purpose of a long-term project, called the Flyspeck project, is to give a formal proof of the Kepler conjecture.  This section makes some preliminary remarks about
formal proofs and gives a general overview of the current status of this project.

\subsection{Formal proof}

A formal proof is a proof in which every logical inference has been checked, all the way back to the foundational axioms of mathematics.  No step is skipped no matter how obvious it may be to a mathematician.  
A formal proof may be less intuitive, and yet is less susceptible to logical
errors.  Because of the large number of inferences involved, a computer is used to check the steps of a formal proof.

It is a large labor-intensive
undertaking to transform a traditional proof
into a formal proof.   The first stage is to expand the traditional proof in greater detail.  This stage fills in steps that a mathematician would regard as obvious,  works out arguments that the original proof leaves to the reader,  and supplies the assumed background knowledge.  
In a final stage, the detailed text is transcribed into a computer-readable format inside
a computer proof assistant.  The proof assistant  contains  mathematical axioms, logical rules of inference, and a collection of previously proved theorems. 
It validates each new lemma by stepping through each inference.  No other currently available technology is able to provide levels of certification of a complex mathematical proof that is remotely comparable to that available by formal computer verification.  A general overview
of formal proofs can be found at~\cite{gonthier:2008:formal}, \cite{Hales:2008:formal}, \cite{Harrison:2008:formal}.

Proof assistants differ in detail in the way they treat the formalization
of  a theorem that is itself a computer verification (such as
the proof of the four color theorem or the proof of the Kepler conjecture).
In general, a formal proof of a computer verification can be viewed
as a formal proof of the correctness of the computer code used in the
verification.  That is, the formal proof certifies that the the code
is a bug-free implementation of its specification.

\subsection{Formal proof of the Kepler conjecture}
\label{sec:intro}

As mentioned above, the purpose of the Flyspeck project is to give a formal proof of the Kepler conjecture.  (The project name {\it Flyspeck} comes from the acronym {\it FPK}, for the {\it Formal} {\it Proof} of the {\it Kepler} conjecture.)  This is the most complex formal proof ever undertaken.  We estimate that it may take about twenty work-years to complete this formalization project.  

The Flyspeck project is introduced in the article~\cite{hales:DSP:2006:432}.
The project page gives the latest developments~\cite{website:FlyspeckProject}.
The project is now at an advanced stage; in fact, we estimate that the project is now about half-way complete.  
One of the main purposes of this article is to present a summary of the current status of this project.

In the original proof of the Kepler conjecture, there was a long mathematical text
and  three major pieces of computer code.
The written part of the proof has been substantially revised with aims of the Flyspeck project in mind.   Section~\ref{sec:blueprint} compares this revised text with the original.  There is now a good match between the mathematical background assumed in the text and the mathematical material that is available in the proof assistant HOL Light.  Section~\ref{sec:ordinary} describes the current level of support in HOL Light for the formalization of Euclidean space and measure theory.
 In the years following the publication of the original proof, S. McLaughlin has  reworked and largely rewritten the entire body of code in a form that is more transparent and more amenable to formalization.  Section~\ref{sec:code} points out some difficulties
in verifying the computer code in its original form and
documents the reimplementation.

There have been three Ph.D. theses written on the Flyspeck
project, one devoted to each of the three major pieces of computer code.
The first piece of computer code uses interval arithmetic to verify nonlinear inequalities.  R. Zumkeller's thesis~\cite{roland-thesis} develops nonlinear inequality proving inside the proof assistant Coq.  Section~\ref{sec:zumkeller} gives an example of this work.
The second piece of computer code enumerates all tame graphs.
(The definition of tameness is rather intricate; its key property is that the set of tame graphs 
includes all graphs that give a potential counterexample to the conjecture.)
G. Bauer's thesis, together with subsequent 
work with T. Nipkow, completes the formal
proof of the enumeration of tame graphs~\cite{NipkowBS-IJCAR06}.
Section~\ref{sec:graph} gives a summary of this formalization project.
The third piece of computer code generates and runs some $10^5$ linear
programs.   These linear programs show that none of the potential counterexamples
to the Kepler conjecture are actual counterexamples.  S. Obua's thesis develops the
technology to generate and verify the linear programs inside the proof assistant Isabelle~\cite{obua:phd}. 
Section~\ref{sec:lp} describes this research.

The ultimate aim is to develop a complete formal proof of the Kepler
conjecture within a single proof assistant.  Because of the scope of 
the problem and the number of researchers involved, different proof
assistants have been used for different parts of the proof:
HOL Light for background
in Euclidean geometry and the text,
Coq for
nonlinear inequality verification, and
Isabelle/HOL for graph enumeration and linear programming.  This raises the issue of how
to translate a formal proof automatically from one proof assistant to another.
Implementations of automated translation among the
proof assistants HOL, Isabelle, and Coq can be found 
at~\cite{obua:import}, \cite{McLaughlin:2006:IJCAR}, \cite{wiedijk:encoding},
\cite{695027}.

\section{Blueprint edition of the Kepler conjecture}
\label{sec:blueprint}

The {\it blueprint edition} of the proof of the Kepler conjecture is a second-generation proof that contains far more explicit detail than the original proof.  The blueprint edition is available at \cite{hales:2008:blueprint}, \cite{hales:2008:collection}.   Many proofs have been significantly simplified and systematized.   It has been written in a manner to permit easy formalization.  As its name might suggest, this version is intended as a blueprint for the construction of a formal proof.
This section 
compares the blueprint edition with the original.

\subsection{Lemmas in elementary geometry}

A collection of about 200 lemmas that can be expressed in elementary terms has been extracted from 
the original proof and placed in a separate collection~\cite{hales:2008:collection}.  This has
several advantages.  First of all, these lemmas, although elementary, are precisely the parts of the
original proof that put the greatest burden on the reader's geometrical intuition.  (Many of these
lemmas deal with the existence or non-existence of configurations of several points in $\ring{R}^3$ subject to
various metric constraints.)  Also, the lemmas can be stated without reference to the Kepler conjecture and all of the
machinery that has been introduced to give a proof.
Finally, the proofs of these lemmas rely on similar methods and are best considered
together~\cite{1271687}.  Section~\ref{sec:ordinary} on {\it Enhanced Automation} gives an approach to proving the lemmas in this collection.

These lemmas  can be
expressed in the first-order language of the real numbers; that is, they can be expressed in
the syntax of first-order logic with equality (allowing quantifiers $\forall x$, $\exists x$ with variables
running over the real numbers), the real constants $0$, $1$, and ring operations $(+)$, $(-)$,
$(\cdot)$ on $\ring{R}$.  In fact, L. Fejes T\'oth's statement of the Kepler conjecture as
an optimization problem in a finite number of variables can itself be expressed in the
first-order language of the real numbers~\cite{Toth:1972:Lagerungen}.  (For this, the truncation used by L. Fejes T\'oth must be modified slightly so that the truncated Voronoi cells are polyhedra.)  Thus, it should come as no surprise that many
of the intermediate lemmas in the proof can also be expressed in this manner.

For example, consider the statement asserting the existence of a circumcenter of a triangle: 
if three points in the plane are not collinear, then there exists a
point in the plane that is equidistant from all three.  This can be expressed in elementary terms
as follows:  for every $(x_1,y_1)$, $(x_2,y_2)$, $(x_3,y_3)$, if there do not exist $t_1, t_2$, and $t_3$
for which $t_1 (x_1,y_1) + t_2 (x_2,y_2)+ t_3 (x_3,y_3) = (0,0)$ and $t_1+t_2+t_3=1$, then there exists
$(x,y)$ such that 
$$
  (x-x_1)^2 + (y-y_1)^2 = (x-x_2)^2 + (y-y_2)^2 = (x-x_3)^2 + (y-y_3)^2. 
$$

\subsection{Background material}

There is now a close
match between what the blueprint edition assumes as background and what the proof assistant {\it HOL Light}
provides, as described in Section~\ref{sec:ordinary}.  The blueprint edition develops 
substantial background
material in trigonometry, measure and integration, hypermaps, and fans (a geometric realization of a hypermap).   It turns out
that only a small number of {\it primitive} volumes need to be computed for the proof of the Kepler
conjecture. These primitives include the volume of a ball, a tetrahedron, and right circular cone.  
No line or surface integrals are required.

A hypermap consists of  three permutations $e,n,f$ (on a finite set $D$) 
that compose to the identity $e\circ n\circ f= I$.  A hypermap is the
combinatorial structure used by Gonthier in his formal proof of the four color theorem.  In 2005,
the proof of the Kepler conjecture was rewritten in terms of hypermaps, because it is better suited
for formal proofs than planar graphs.  

The Jordan curve theorem (JCT) has been formalized as a step in the Flyspeck project~\cite{Hales:2007:jordan}.  In fact, something weaker than the JCT
suffices.  The project only uses the JCT for curves on
the surface of a unit sphere consisting of a finite number of arcs of great circles; that is, a 
spherical polygonal version of the JCT.  In the proof of the Kepler conjecture, the combinatorial structure of
a cluster of spheres is encoded as a hypermap.  An Euler characteristic calculation, based on the JCT, shows that
these hypermaps are planar.  The planarity of these hypermaps is a crucial property that is used in the enumeration
of tame graphs (Section~\ref{sec:graph}).
The background chapter in the blueprint edition contains detailed proofs of these facts.

The blueprint edition contains several introductory essays that introduce 
the main concepts in the proof, including the algorithms implemented by the computer code.

The formalization of the blueprint text started in
2008 with work of J. Rute (CMU) and the Hanoi Flyspeck group. The current members of this group are {\vn Tr\`\acircumflex n Nam Trung}, {\vn Nguy\~\ecircumflex n T\'\acircumflex t Th\'\abreve ng}, {\vn Ho\`ang L\ecircumflex\ Tr\uhorn\`\ohorn ng}, {\vn Nguy\~\ecircumflex n Quang Tr\uhorn\h\ohorn ng}, {\vn V\~u Kh\'\abreve c K\h y}, {\vn Nguy\~\ecircumflex n Anh T\acircumflex m}, {\vn Nguy\~\ecircumflex n Tuy\ecircumflex n Ho\`ang}, {\vn Nguy\~\ecircumflex n \DJ\'\uhorn c Ph\uhorn \ohorn ng}, {\vn V\uhorn \ohorn ng Anh Quy\`\ecircumflex n}, {\vn Phan Ho\`ang Ch\ohorn n}, and managed by {\vn T\d a Th\d i Ho\`ai An}.  The blueprint formalization is still at an early stage.

\section{Formalizing the ordinary mathematics}
\label{sec:ordinary}

This section describes some of the
work on formalizing Euclidean space and measure theory, and the development of
further proof automation, which should be useful in this endeavor.

The computer code in Flyspeck has so far received the lion's share of the
formal effort. This is entirely reasonable since there are, or at least were,
real questions about the feasibility of reproducing these results in a formal
way. However, the Flyspeck proof includes a large amount of `ordinary'
mathematics, which also needs to be formalized. Here we are on fairly safe
ground in principle, because by now we understand the formalization of
mainstream mathematics quite well~\cite{wiedijk-17}. It is safe to predict that
this formalization can be done, and we can even hope for a reasonably accurate
estimate of the effort involved. Nevertheless, the formalization is certainly
non-trivial and will require considerable work.

\subsection{Formalizing Euclidean space}

Much of our work has been devoted to developing a solid general theory of
Euclidean space $\real^N$ \cite{harrison-euclidean}. For Flyspeck, we
invariably just need the special case $\real^3$. While some concepts, e.g.
vector cross products, are specific to $\real^3$, most of the theory has been
developed for general $\real^N$ so as to be more widely applicable. The theorem
prover HOL Light \cite{harrison-demo} is based on a logic without dependent
types, but we can still encode the index $N$ as a type (roughly, an arbitrary
indexing type of size $N$). This means that theorems about specific sizes like
$3$ really are just type instantiations of theorems for general $N$ stated with
polymorphic type variables. The theory contains the following:

\begin{itemize}

\item Basic properties of vectors in $\real^N$, linear operators and matrices,
dimensions of vector subspaces and other bits of linear algebra. For example,
the following is a formal statement of the theorem that a square matrix $A'$ is
a left inverse to another one $A$ iff it is a right inverse. Note that the
double use of the same type variable $N$ constrains the theorem to square
matrices:

\begin{\sz}
\begin{alltt}
|- \(\forall\)A:real^N^N A':real^N^N. (A ** A' = mat 1) \(\Iff\) (A' ** A = mat 1)
\end{alltt}
\end{\sz}

\item Metric and topological notions like distances, open sets, closure,
compactness and paths. Some of these are very general, others are more specific
to Euclidean space. Some results include the Heine-Borel theorem, the Banach
fixed-point theorem and Brouwer's fixed-point theorem. The following is a
formal statement that continuous functions preserve connectedness.

\begin{\sz}
\begin{alltt}
|-  \(\forall\)f:real^M->real^N s. f continuous_on s \(\And\) connected s 
         \(\Imp\) connected(IMAGE f s)
\end{alltt}
\end{\sz}

\item Properties of convex sets, convex hulls, cones etc. Results include
Helly's theorem, Carath\'eodory's theorem, and various classic results about
separating and supporting hyperplanes. The following states a simpler but not
entirely trivial result that convex hulls preserve compactness.

\begin{\sz}\begin{alltt}
|- \(\forall\)s:real^N->bool. compact s \(\Imp\) compact(convex hull s)
\end{alltt}\end{\sz}

\item Sequences and series of vectors and uniform convergence, Fr\'echet
derivatives and their properties, up to various forms of the inverse function
theorem, as well as specific 1-dimensional theorems like Rolle's theorem and
the Mean Value Theorem. Here is the formal statement of the chain rule for
Fr\'echet derivatives.

\begin{\sz}\begin{alltt}
|- \(\forall\)f:real^M->real^N g:real^N->real^P f' g'.
        (f has_derivative f') (at x) \(\And\)
        (g has_derivative g') (at (f x))
        \(\Imp\) ((g o f) has_derivative (g' o f')) (at x)
\end{alltt}\end{\sz}

\end{itemize}

\subsection{Formalizing measure theory}

Although the basic Euclidean theory is an important foundation, and many of the
concepts like `convex hull' are used extensively in the Flyspeck mathematics,
perhaps the most important thing to formalize is the concept of volume.

We define integrals of general vector-valued functions over subsets of
$\real^N$, using the Kurzweil-Henstock gauge integral definition. We develop
all the usual properties such as additivity and the key monotone and dominated
convergence theorems. We also develop a theory of {\em absolutely} integrable
functions, where both $f$ and $|f|$ are gauge integrable; this is known to
coincide with the Lebesgue integral. Here is a formal statement of the simple
theorem that integration preserves linear scaling:

\begin{\sz}
\begin{alltt}
|- \(\forall\)f:real^M->real^N y s h:real^N->real^P.
        (f has_integral y) s \(\And\) linear h \(\Imp\) ((h o f) has_integral h(y)) s
\end{alltt}
\end{\sz}

Using this integral applied to characteristic functions, we develop a theory of
(Lebesgue) measure, which of course gives volume in the 3-dimensional case. The
specific notion `measure zero' is formalized as {\tt negligible}, and we also
have a general notion of a set having a finite measure, and a function {\tt
measure} to return that measure when it exists. For example, this is the basic
additivity theorem:

\begin{\sz}
\begin{alltt}
|- \(\forall\)s t. measurable s \(\And\) measurable t \(\And\) DISJOINT s t
         \(\Imp\) measure(s UNION t) = measure s + measure t
\end{alltt}
\end{\sz}

We have proved that various `well-behaved' sets such as bounded convex ones and
compact ones, or more generally those with negligible frontier (boundary) are
measurable, e.g.

\begin{\sz}
\begin{alltt}
|- \(\forall\)s:real^N->bool. bounded s \(\And\) negligible(frontier s) \(\Imp\) measurable s
\end{alltt}
\end{\sz}

The main lack at the moment is a set of results for actually computing the
measures of specific sets, as needed for Flyspeck. We can evaluate most basic
1-dimensional integrals by appealing to the Fundamental Theorem of Calculus,
but we need to enhance the theory of integration with stronger Fubini-type
results so that we can evaluate multiple integrals by iterated one-dimensional
integrals. This work is in progress at the time of writing.

\subsection{Enhanced automation}

Using coordinates, many non-trivial geometric statements in $\real^3$, or other
Euclidean spaces of specific finite dimension, can be reduced purely to the
elementary theory of reals. This is known to be decidable using quantifier
elimination \cite{tarski-decision,collins,hormander-pdo2}. However, in practice
this is often problematic because quantifier elimination for nonlinear formulas
is inefficient. The problem is particularly severe if we want to have any kind
of {\em formal} proof, as we do in Flyspeck, since producing such a proof
induces further slowdowns \cite{mahboubi-hormander}, \cite{mclaughlin-harrison}. With
this in mind, we have explored a different approach to the case of purely
universally quantified formulas \cite{harrison-sos}, based on ideas of Parrilo
\cite{parrilo-semidefinite}. This involves reducing the initial problem to
semidefinite programming, solving the SDP problem using an external tool and
reconstructing a `sum-of-squares' (SOS) certificate that can easily be formally
checked.

For example, suppose we wish to verify that if a quadratic equation $a x^2 + b
x + c = 0$ has a real root, then $b^2 \geq 4 a c$. Using the SDP solver we find
an algebraic certificate $b^2 - 4 a c = (2 a x + b)^2 - 4 a (a x^2 + b x + c)$,
from which the required fact follows easily: $(2 a x + b)^2 \geq 0$ because it
is a square, and $4 a (a x^2 + b x + c) = 0$ because $x$ is a root, and so we
deduce $b^2 - 4 a c \geq 0$. This method seems very useful for automating
routine nonlinear reasoning in a way that is easy and quick to formally verify,
so that we don't have to rely on the correctness of a complicated program. It
is even capable of solving the coordinate forms of some of the simpler Flyspeck
inequalities directly, though it seems unlikely to be competitive with
customized nonlinear optimization methods as described in Section~\ref{sec:zumkeller}. For
example, one simple Flyspeck inequality is the following, which after being
reduced to a real problem with 9 variables (three coordinates for each point)
is solved by SOS in a second:
$$ \|u - v\| \geq 2 \,\And\, \|u - w\| \geq 2 \,\And\, \|v - w\| \geq 2 \,\And\,
   \|u - v\| < \sqrt{8}
   \Imp \|w - (u + v)/2\| > \|u - v\|/2.
$$

A quite different approach to geometric theorem proving is to work in the
setting of a general real vector space or normed real vector space. In this
case, other decision methods are available \cite{solovay-jointpaper}. In
particular, one of these decision procedures that we have implemented in HOL
Light can sometimes handle simple forms of spatial reasoning in a purely
`linear' way, and so be much more efficient than the direct reduction to
coordinates, even if we do in fact have a specific dimension in mind. One
real example from formalizing complex analysis is the following in $\real^2$:
$$ |\|w - z\| - r| = d \And \|u - w\| < d/2 \And \|x - z\| = r
   \Imp d/2 \leq \|x - u\|.
$$

\section{Standard ML reimplementation of code}
\label{sec:code}

This section describes a reimplementation of the computer code 
used in the proof of the Kepler conjecture. 
The code has been substantially
redesigned to avoid various difficulties
with the original implementation.

\subsection{Code}

The original proof of the Kepler conjecture relies significantly on
computation. Computer code is used extensively and is central both to
the correctness of the result and to a thorough understanding of the
proof.

  There are four major difficulties with understanding and verifying
the original code base. The first and most glaring difficulty is
simply the amount of code. At the website~\cite{website:Hales:1998:Code}
that posts the code for the original proof
there are well over 50,000 lines of programs in Java, C++, and
Mathematica (among others). This represents only the calculations that Hales did himself.
Samuel Ferguson also completed many of the calculations with an entirely
different code base\footnote{There is a large amount of
code copying in Ferguson's code, resulting in a much larger code base.
The number of distinct lines is difficult to measure} of 137,000 lines of C. By
contrast, the proof of the four color theorem by Robinson \textit{et
al.}~\cite{Robertson:1997:JCTB} is less than 3,000 lines of C.

The second difficulty is in the organization of the code. The
calculations were done over the space of four years and involved
thousands of executions of a multitude of independent programs. 
Section~\ref{sec:intro} identifies  three main computational
tasks: tame graph generation, linear program
bounding, and nonlinear inequality verification.  Each of these main
tasks consists of several subtasks.  For example, verifying the
inequalities required dozens of relatively complicated preprocessing
phases where second derivatives of the relevant functions were bounded
over fixed domains.  As another example, many linear programs were
solved only after a branch and bound period which were recorded in
voluminous log files.  In an attempt to organize the complex 
web of calculations, Hales
devised a labeling scheme to uniquely identify the calculations.
However, even now some computations relied upon by the proof are
difficult to find in the original source code.  To locate, for
instance, computation ``\calc{821707685}'' from the original
proof~\cite[p.159]{Hales:2006:DCG}, one can search on the
website~\cite{website:Hales:1998:Code} in vain.  While records of the
computations were made, it is not always an easy matter to find them
without guidance from the authors.

The third difficulty lies in the complexity of the implementation. For
instance, the software developed to prove the inequalities upon which
the original proof rests is relatively complicated. Processing power
at the time (1994-1998) was just barely capable of completing the
computations requested. To keep the length of execution to days or
weeks instead of months or years, the code is extensively
optimized. The optimizations were often implemented without comment in
the source and in some cases are difficult to understand.

The fourth difficulty is that
the original code uses C and C++ to carry out \emph{interval
arithmetic} calculations based on \emph{floating point arithmetic}. In
the process, it explicitly sets the IEEE 754~\cite{IEEE:1985:IEE754}
rounding modes on the processor's floating point unit.  While floating
point is desirable for its speed, there are difficulties with using
floating point for software that requires a very high level of rigor
such as that supporting mathematical proof.  The first is that
reasoning about floating point instructions requires a relatively deep
understanding of the machine architecture~\cite{Monniaux:2008:TOPLAS}.
For instance, setting the rounding mode changes the state of the
processor itself. Such an instruction has a global effect on all
subsequent floating point computations.  In the original code base the
rounding modes are explicitly changed at least 400 times.  Moreover,
compilers, libraries, and even processors are notorious for unsound
implementations of the 754 standard.  


\subsection{Reimplementation}
\label{sec:sean}

In 2004 we decided to reimplement the original code
base.  We decided that the new implementation should not require 
floating point numbers and rounding modes.  Though speed
was important, we wanted the code to be independent of any particular
interval arithmetic implementation.  This meant we could use a fast
floating point implementation of interval arithmetic for our daily
work, but could use a slower but more trustworthy implementation such
as MPFI~\cite{Revol:2005:MPFI} to double check important computations.
We also wanted to organize the new implementation such that any of the
many computations upon which the proof relies could be evaluated from
a single interface.  This would allow Flyspeck developers to 
find and easily check the text of the proof during the formalization process.
Finally, we wished to bring the computational aspects of the
Kepler conjecture closer to the level of simplicity and clarity
necessary for formalization by a proof assistant.  We began this work
in the spirit of Robinson \textit{et
al.}~\cite{Robertson:1997:JCTB}, which simplified the original code of
Appel and Haken~\cite{Appel:1986:FourColor}, and was used by Gonthier
to construct the fully formal proof~\cite{gonthier:2008:formal} in the
Coq proof assistant.

We chose Standard ML for the reimplementation for a number of reasons.
It has a formal definition~\cite{Milner:1990:SML}, and thus programs
have a meaning apart from the particular compiler used.  It has an
efficient compiler, named MLton~\cite{website:MLton}.  (Our
reimplementation runs between 50\% and 200\% the speed of the original
implementation compiled with GCC.)  MLton has the ability to use
external libraries written in languages other than SML with relative
ease.  This allowed us, from one programming environment, to control multiple linear programming solvers,
interval arithmetic implementations, and nonlinear optimization
packages.
SML has an expressive module system, and thus it was simple to write
our code with respect to an abstract type of interval arithmetic. 
Thus we could use multiple independent implementations
with ease.  As of 2008, most of the code has been completely
rewritten in SML and is executable
from a single command-line program.  The code is freely available
at the project website~\cite{McLaughlin:2008:KeplerCode}.

In the original implementation, the myriad computations were done with
many different programs written in a half dozen programming languages.
The results of these computations are not always easy to find or
interpret.  Now all the computations are executed from the same
source, with organized output.  In addition to giving us added
confidence that the original computations were sound, we have a fairly
complete suite of software support for the Flyspeck project. We are
now in the process of organizing and reevaluating the thousands of
computations upon which the proof depends.

\section{Proving nonlinear inequalities with Bernstein bases}
\label{sec:zumkeller}

The hardest computational part of the original proof of the Kepler conjecture is the verification of a list of about a thousand nonlinear inequalities. This section presents a technique aimed at proving
them, based on polynomial approximation and Bernstein bases. We feel that
this approach better fits the requirements of formal proof, as outlined in
Section~\ref{sec:sean}.  We hope to refine the method to cover all Flyspeck
inequalities.

We exhibit the method on a single inequality~\calc{586468779}.
The original proof contains the following definitions \cite{sp1}:
\begin{align*}
  \pt &\coloneq - \frac \pi 3 + 4 \arctan \frac{\sqrt 2}5 \\
\doct &\coloneq \frac {\pi - 4 \arctan \frac{\sqrt 2}5}{2 \sqrt 2}\\
\Delta(y) &\coloneq \frac 1 2
  \left|
  \begin{array}{ccccc}
0 & 1 & 1 & 1 & 1 \\
   1 & 0 & y_3^2 & y_2^2 & y_1^2 \\
   1 & y_3^2 & 0 & y_4^2 & y_5^2 \\
   1 & y_2^2 & y_4^2 & 0 & y_6^2 \\
   1 & y_1^2 & y_5^2 & y_6^2 & 0
  \end{array}
  \right|\\
\asolid_0(y) &\coloneq y_1 y_2 y_3 + \sfrac 1 2 (
y_1^2 y_2 + y_1 y_2^2 + y_1^2 y_3 + y_2^2 y_3 + y_1 y_3^2 \\
&\qquad\qquad\qquad {} + y_2 y_3^2 - y_1 y_4^2 - y_2 y_5^2 - y_3 y_6^2)\\
\asolid_1(y) &\coloneq \asolid_0 (y_1, y_5, y_6, y_4, y_2, y_3)\\
\asolid_2(y) &\coloneq \asolid_0 (y_2, y_4, y_6, y_5, y_1, y_3)\\
\asolid_3(y) &\coloneq \asolid_0 (y_4, y_5, y_3, y_1, y_2, y_6)\\
\gamma(y) &\coloneq 
- \frac \doct 6 \sqrt {\Delta(y)} + \frac 2 3 \sum_{i=0}^3 \arctan \frac
{\sqrt{\Delta(y)}} {\asolid_i(y)}.
\end{align*}

The statement of the inequality is:
\begin{equation}
\forall y \in [2,2.51]^6.\; \gamma(y) \le \pt. \label{gamma-pt}
\end{equation}
Define the difference of two intervals by $[a_1,b_1]-[a_2,b_2] = [a_1-b_2,b_1-a_2]$.
Interval arithmetic is used to prove the inequalities in the original proof.
It
suffers from the \emph{dependency problem}: the minimum and maximum of the
formula $x-x$ are overestimated because $[a,b] - [a,b] = [a-b,b-a]$, although
$x-x$ is clearly $0$. Subdividing $[a,b]$ into $[a,\sfrac{a+b}2]$ and
$[\sfrac{a+b}2,b]$, and then re-evaluating the formula yields an improved
result. However, depending on the problem, the number of required subdivisions
can be very large. This is why checking some inequalities takes a very long
time.

\subsection{From geometrical functions to polynomials}
Fortunately, better methods than interval arithmetic are available, if the
function under consideration is polynomial. A quick look at $\gamma$ tells us
that \eqref{gamma-pt} is not polynomial, since it has occurrences of
$\sqrt\cdot$, $1/\cdot$, and $\arctan$. Can it nevertheless be reduced
to a polynomial problem? Two strategies come to mind:

First, algebraic laws such as $\sqrt a \le b \Leftrightarrow a \le b^2$ (if $b
\ge 0$) and $\frac a b \le c \Leftrightarrow a \le bc$ (if $b>0$) can often be
used to eliminate occurrences of $\sqrt\cdot$ and $1/\cdot$. The list of
trigonometric identities is endless. For our example, Vega's rule $\arctan a +
\arctan b = \arctan \frac {a + b} {1 - ab}$ seems useful. Unfortunately, this
technique quite often yields huge expressions that are difficult to deal with by
virtue of their sheer size. Also, an algebraic transformation to a polynomial problem may
simply be impossible (we suspect that this is the case for \eqref{gamma-pt}).

A second technique is based on replacing $\gamma$ with a polynomial $g$ that
dominates it, but is still smaller than $\pt$. Clearly, if there exists a $g$
such that
\begin{equation}
\forall y \in [2,2.51]^6.\; \gamma(y) \le g(y) \label{gamma-g}
\end{equation}
and
\begin{equation}
\forall y \in [2,2.51]^6.\; g(y) \le \pt, \label{g-pt}
\end{equation}
then by transitivity \eqref{gamma-pt} holds.

Such a polynomial $g$ can be obtained by replacing $\sqrt{\cdot}$, $1/\cdot$ and
$\arctan$ with polynomial approximations. We only need to ensure that we use
upper approximations for positive occurrences and lower approximations for
negative ones. Only occurrences whose arguments contain variables need to be
replaced, since e.g. $\sqrt 2$ is a (constant) polynomial itself.

In the definition of $\gamma$ the function $\arctan{}$ occurs positively, so it
is replaced by an upper approximation $\overline{\arctan}$. The term
$\frac{\sqrt{\Delta(y)}} {\asolid_i(y)}$ is first unfolded to $
{\sqrt{\Delta(y)}}~\cdot~\frac1{\asolid_i(y)}$. Both the square root and
reciprocal occur positively again here, so they can be replaced by upper
approximations $\overline\sqroot$ and $\overline\rcp$, respectively. This yields
$\overline{\arctan} (\overline\sqroot(\Delta(y)) \cdot
\overline\rcp(\asolid_i(y)))$ in all four summands. There remains only 
$\sqrt\cdot$ occurring negatively after $-\frac \doct 6$, which is to be
replaced by a lower approximation $\underline\sqroot$.

We choose the following approximations:
\begin{align*}
\overline{\arctan}(t) &\coloneq \arctan \frac{\sqrt 2}5 + \frac {25} {27}\left(t - \frac{\sqrt 2}5\right)\\
\overline\rcp(t) &\coloneq \frac 1 4 - \frac {37 t} {1600} + \frac {t^2} {1000}- \frac{13 t^3} {640000} + \frac {t^4} {6400000}\\
\underline\sqroot(t) &\coloneq 8 \sqrt 2 + \frac 3 {64 (\pi - 4 \arctan \frac{\sqrt 2}5)} (t - 128)\\
\overline\sqroot(t) &\coloneq 8 \sqrt 2 + \frac 1 {16 \sqrt 2}(t -128).
\end{align*}

These approximations are valid on the domain~\eqref{gamma-pt}. For example,
$$\forall t \in\Delta ([2,2.51]^6).\; \sqrt{t} \le \overline\sqroot(t).$$ This
can be established by elementary means, knowing that $\Delta ([2,2.51]^6)
\subseteq [128;501]$. The latter can be shown automatically by the method
outlined in the next subsection.

In summary, we arrive at the following definition of $g$:
\begin{align*}
g(y) &\coloneq 
- \frac \doct 6 \underline\sqroot (\Delta(y)) + \frac 2 3 \sum_{i=0}^3 
\overline{\arctan} (\overline\sqroot(\Delta(y)) \cdot
\overline\rcp(\asolid_i(y))).
\end{align*}
In form, it is almost identical to the definition of $\gamma$.
Our construction of $g$ therefore ensures \eqref{gamma-g}. Moreover, the
approximations were chosen (using polynomial interpolation) in a way such that
\begin{equation}
\gamma (2,2,2,2,2,2) = g (2,2,2,2,2,2) = \pt. \label{eq-gamma-g-pt}
\end{equation}
This is important, because otherwise~\eqref{g-pt} cannot hold.

\subsection{Bounding polynomials}
\label{bernstein}
In order to prove~\eqref{gamma-pt}, it remains to be shown that $g(y) \le \pt$.
This can be done with the help of Bernstein polynomials. We briefly outline the
case of a single variable $x$ here.

The $i$th Bernstein basis polynomial of order $k$ is defined as
$$\bstein^k_i(x) \coloneq {\binom{k}{i} }
x^i(1-x)^{k-i}.$$
For a polynomial $p$ and a vector $b \in \mathbb R^k$, if
$$p(x) = \sum_{i=0}^k b_i \cdot \bstein^k_i(x),$$
then $b$ is called the \emph{Bernstein representation} of $p$. In this case
$$\forall x \in [0;1].\; p(x) \le \max_i b_i.$$

This property is tremendously useful: it gives us an upper bound on $p$, namely
the largest coefficient of $p$'s Bernstein representation. By a change of
variable we can reduce any interval to $[0,1]$. The generalization to the
multivariate case is straightforward~\cite{garloff}, \cite{roland-thesis}.

In order to bound a polynomial it thus suffices to convert it into Bernstein
representation. This can be done by a matrix multiplication (the Bernstein basis
of order $k$ forms a basis of the vector space of all polynomials of degree up
to $k$). For practical purposes it is however crucial to use a more efficient
algorithm (cf. \cite{garloff}, \cite{roland-thesis}).

Note that $g$ contains irrational coefficients. This is a consequence of
requirement~\eqref{eq-gamma-g-pt} and cannot be avoided. However, we were able
to choose the approximation polynomials in a way such that the transcendental
parts can be factored out (hence the occurrence of $\pi - 4 \arctan \frac{\sqrt
  2}5$ in the definition of $\underline\sqroot$). As can be easily checked with
symbolic algebra software, the polynomial $p(y) \coloneq \sqrt 2 (g(y) - \pt)$
has rational coefficients! It can thus be converted to a Bernstein
representation without rounding, using the algorithm presented in
\cite{roland-thesis}. With this method the only divisions are by powers of 2,
which can be efficiently represented using dyadic numbers.

The polynomial $p$ consists of 12945 monomials and has total degree
18. A prototype implementation in Haskell returns $0$ as the maximum
for $p$ in about half a minute. Thus $\sqrt 2 (g(y) - \pt) \le 0$ and
$g(y) \le \pt$.

\section{Tame graph enumeration}

\label{sec:graph}

Tame graphs are particular plane graphs that represent potential
counterexamples to the Kepler conjecture. The \emph{Archive} is a list of over 5000 plane
graphs.  The original proof generates the Archive with the help of a Java program that
enumerates all plane graphs. Tameness is defined
in Section~18 and the enumeration is sketched in Section~19 of
\cite{Hales:2006:DCG}.
This section sketches the formally machine-checked proof of Claim~3.13 and
Theorem 19.1 in the original proof~\cite{Hales:2006:DCG}:
\begin{thm}\label{Archive:complete}
Any tame plane graph is isomorphic to a graph in the Archive.
\end{thm}
There are two potential reasons why an error in the original proof of this theorem might have gone undetected:
the publications only sketch the details of the enumeration, and the
referees only made a passing glance at the implementation, consisting of more than 2000
lines of Java.

We recast the Java program for the enumeration of all tame graphs in
logic, proved its completeness with the help of an interactive theorem prover,
ran it, and compared the output to the
Archive.  It turns out that the original proof was right, the Archive is complete,
although redundant (there are at most 2771 tame graphs).
Doing all this inside a logic and a theorem prover requires two things:
\begin{itemize}
\item The logic must contain a programming language.
We used Church's \emph{higher-order logic} (HOL) based on $\lambda$-calculus,
the foundation of functional programming. Programs in HOL are simply
sets of recursion equations, i.e.\ pure logic.
\item The programming language contained in the logic must be efficiently
executable and such executions must count as proofs. The theorem prover that we used, Isabelle/HOL~\cite{LNCS2283} fulfills this criterion. If all functions
that appear in a term are either data, e.g.\ numbers, or functions defined
by recursion equations, Isabelle/HOL offers the possibility to evaluate this term $t$ in one big and relatively efficient step to a value $v$, giving rise to the theorem $t = v$.
\end{itemize}
The enumeration of all tame graphs generates 23 million plane graphs ---
hence the need to perform massive computations in reasonable time.

Now we give a top-level overview of the formalization and proof of
completeness of the enumeration of tame graphs in HOL. For details
see~\cite{NipkowBS-IJCAR06}. The the complete machine-checked proof, over
17000 lines, is available online in the Archive of Formal Proofs at
\url{afp.sf.net}~\cite{BauerN-AFP06}.

\subsection{Plane graphs}

Following the original proof, we represent finite, undirected, plane graphs as lists
(= finite sets) of faces and faces as lists of vertices. Note that by
representing faces as lists they have an orientation. The enumeration of
plane graphs requires an additional distinction between \emph{final}
and \emph{non-final} faces. Hence a face is really a pair of a list
of vertices and a Boolean.
A plane graph is \emph{final} iff each of its faces is.
In final graphs we can ignore the Boolean component of the faces.

\subsection{Enumeration of plane graphs}

The original proof characterizes plane graphs by an executable enumeration and sketches a
proof of completeness of this enumeration. We have followed the original proof and taken
this enumeration as the definition of planarity.  The enumeration of plane
graphs in the original proof proceeds inductively: you start with a seed graph with two faces, the
final outer one and the (reverse) non-final inner one. If a graph contains a
non-final face, it can be subdivided into a final face and any number of
non-final ones.  Because a face can be subdivided in many ways, this process
defines a tree of graphs. By construction the leaves must be final graphs,
and they are the plane graphs we are interested in: any plane graph of $n$
faces can be generated in $n-1$ steps by this process, adding one (final)
face at a time. For details see~\cite{Hales:2006:DCG} or
\cite{NipkowBS-IJCAR06}.

The enumeration is parameterized by a natural number $p$ which controls the
maximal size of final faces in the generated graphs. The seed graph
\textit{Seed$_p$} contains two $(p+3)$-gons and the final
face created in each step may at most be a $(p+3)$-gon. As a result,
different parameters lead to disjoint sets of graphs.

The HOL formalization defines an executable function
\textit{next-plane$_p$} that maps a
graph to a list of graphs, the successor graphs reachable by subdividing one
non-final face.  The plane graphs are the final graphs reachable from
\textit{Seed$_p$} via \textit{next-plane$_p$} for some $p$.

\subsection{Enumeration of tame graphs}

The definition of tameness in the original proof is already quite close to a direct logical
formulation. Hence the HOL formalization is very close to this. Of course
pictures of graphs had to be translated into formulae, taking implicit
symmetries in pictures into account. We found one simplification: in the
definition of an admissible weight assignment one can drop condition 3
(a condition on the $4$-circuits in graphs)
without changing the set of tame graphs. What facilitated our work
considerably was that a number of the eight tameness conditions 
in the original proof are directly
executable. The details are described elsewhere~\cite{NipkowBS-IJCAR06}.

The enumeration of tame graphs is a modified enumeration of plane graphs
where we remove final graphs that are definitely not tame, and prune the
search tree at non-final graphs that cannot lead to tame graphs anymore.  
The published
description~\cite{Hales:2006:DCG} is deliberately sketchy, and 
the precise formulation of the pruning criteria is based on the
original Java programs.  This
is the most delicate part of the proof because we need to balance
effectiveness of pruning with simplicity of the completeness proof: weak
pruning criteria are easy to justify but lead to unacceptable run times of
the enumeration, sophisticated pruning techniques are difficult to justify
formally. Since computer-assisted proofs are still very laborious,
simplifying those proofs was of prime importance. In the end, the HOL
formalization defines a function \textit{next-tame$_p$}
from a graph to a graph list. It computes the list of plane successor graphs
\textit{next-plane$_p$ g} and post-processes it as follows:
\begin{enumerate}
\item Remove all graphs from the list that cannot lead to tame graphs
because of lower bound estimates for the total admissible weight of the final
graph.
\item Finalize all triangles in all of the graphs in the list
(because every 3-cycle in a tame graph must be a face).
\item Remove final graphs that are not tame from the list.
\end{enumerate}
A necessary but possibly not sufficient check for tameness is used in the last
step. Hence the enumeration may actually produce non-tame graphs. This is
unproblematic: in the worst case a fake counterexample to the Kepler
conjecture is produced, but we do not miss any real ones.

Although we have roughly followed the procedure of the original proof, we have simplified it in
many places. In particular we removed the special treatment of
\textit{Seed$_0$} and \textit{Seed$_1$}, which is a fairly intricate
optimization that turned out to be unnecessary.

The following completeness theorem is the key result:
\begin{thm}\label{thm:TameEnum_comp}
If a tame and final graph $g$ is reachable from
\textit{Seed$_p$} via \textit{next-plane$_p$}
then $g$ is also reachable from \textit{Seed$_p$} via \textit{next-tame$_p$}.
\end{thm}

Each step \textit{next-tame$_p$} is executable and an exhaustive enumeration
of all graphs reachable from a seed graph is easily defined on top of it. We
call this function \textit{tameEnum$_p$}. By definition, tame graphs may
contain only triangles up to octagons, which corresponds to the parameters $p
= 0,\dots,5$.

\subsection{Archive}

In order to build on the above enumeration of all tame graphs without having
to rerun the enumeration, the results of running \textit{tameEnum$_p$} with
$p = 0,\dots,5$ are put into an Archive and isomorphic graphs are
eliminated. This results in 2771 graphs, as opposed to 5128 in the original proof. The
reasons are twofold: there are many isomorphic copies of graphs in the
Archive and it contains a number of non-tame graphs, partly because, for
efficiency reasons, the original proof did not enforce all tameness conditions in its Java
program. The new reduced Archive is also available
online~\cite{BauerN-AFP06}.

Finally we can prove Theorem~\ref{Archive:complete}: if $g$ is tame plane
graph, Theorem~\ref{thm:TameEnum_comp} and the definition of
\textit{tameEnum} tell us that $g$ must be contained in \textit{tameEnum$_p$}
for some $p=0,\dots,5$. Hence it suffices to enumerate \textit{tameEnum$_p$},
$p=0,\dots,5$, and check that, modulo graph isomorphism, the result is the
same as the Archive. This is a proposition that can be proved by executing it
(because the HOL formalization also includes a verified executable test for
graph isomorphism which we do not discuss).

\section{Verifying linear programs}
\label{sec:lp}

This section reports on the current state of the formal verification of
the linear programming part of the proof of the Kepler conjecture.
The results of the linear programming in the original proof are recorded
in several gigabytes of log files.
This section presents the formalization of the generation and bounding of these linear programs 
in the mechanical proof assistant Isabelle~\cite{LNCS2283}. 
A more detailed version of the material presented here can be found in S. Obua's thesis~\cite{obua:phd}.

This formalization relies on the archive of tame graphs. Each tame graph (except the graphs associated with the face-centered cubic and hexagonal close packings) represents a potential counterexample to the Kepler conjecture. 
Each potential counterexample  obeys certain constraints.  The original
proof refutes each potential counterexample by building a linear program
from the constraints and showing that these linear programs are infeasible.

A structure, which we call a \emph{graph system}, makes this precise
in a formal environment.
An instance of a graph system is a tame graph which obeys the constraints listed in the definition of a graph system.
The current formalization does not use all of the constraints of the original proof but only those that do not require branch and bound strategies. Our current notion of graph system can be viewed
as a detailed formalization of what is called \emph{basic linear programs} in~\cite[\S 23.3]{Hales:2006:DCG}. 
Because the current formalization does not capture not all constraints of the original proof, we cannot hope to refute all potential counterexamples. Nevertheless we manage to refute most of them.

We represent tame graphs as \emph{hypermaps}~\cite{gonthier:2008:formal}, \cite{hales:2008:blueprint}. As mentioned in Section~\ref{sec:blueprint}, 
a hypermap is just a finite set $D$ of \emph{darts} together
with three permutations on $D$: the edge, the face, and the node permutation that compose to the identity $e\circ n\circ f = I$.   See Figure~\ref{fig:hypermap}.
This representation greatly simplifies
the axiomatization of a graph system. For a detailed description of how we represent 
hypermaps and for a complete list of all of the axioms of a graph system see~\cite[\S4]{obua:phd}.

\begin{figure}
\begin{center}
\includegraphics[width=12cm]{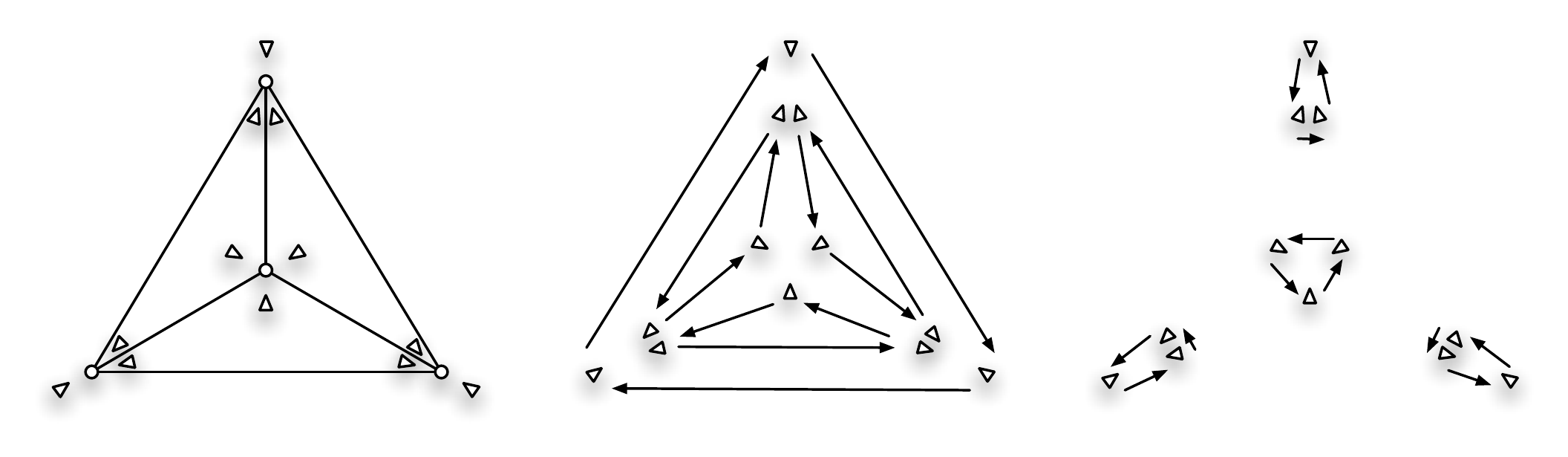}
\end{center}
\caption{A hypermap is a combinatorial structure attached to a planar graph.  A  dart, represented as a small triangle, is place at each face angle of the graph.  In the second frame, the face permutation $f$ cycles through the darts in each face.  In the third frame, the node permutation $n$ cycles through the darts at each node.  The edge permutation $e$ (not shown) is an involution exchanging darts from opposite ends of the edge of a graph. (This figure has ben reproduced from~\cite{Hales:2008:Dodec}.)}
\label{fig:hypermap}
\end{figure}

Figure~\ref{fig:lpapproach} summarizes how we generate and solve the linear programs.
We apply the axioms of a graph system to each tame graph. This results in a large Isabelle theorem which is a conjunction of
linear equalities and inequalities. We then normalize this conjunction
to bring it into the form of a matrix inequality
\begin{equation}
A  x \leq b.
\end{equation}
The entries of $A$ and $b$ are symbolic expressions which contain various real constants (such as $\pi$) that are needed to
define the axioms of a graph system. In order to apply linear programming, 
we need to replace $A$ and $b$ with numerical approximations. We achieve this via a formalization of 
interval arithmetic in Isabelle, and arrive at numerical matrices $A'$, $A''$ and $b'$ for which we have the formally proven 
Isabelle theorems 
\begin{equation}
A' \leq  A  \leq A'', \quad b \leq b'.
\end{equation}
We then apply a simple preprocessing step which gives us formally proven a-priori bounds $x'$ and $x''$ for $x$ such
that 
\begin{equation}
x' \leq x \leq x''.
\end{equation}
It is then possible to obtain a certificate from an external linear programming solver like GLPK (Gnu Linear Programming Kit) that allows us to formally reach a contradiction from $A x \leq b$. The beauty of a certificate is 
that we can use results obtained from an untrusted source (in our case this untrusted source is a 
heavily optimized linear program solver programmed in C) in a trusted and completely mechanically verifiable way.

In this way, we proved the inconsistency of 2565 of the graph systems but failed to prove the inconsistency of the remaining 206. 
This yields a success rate of about $92.5\%$. Future work will extend the notion of graph system and  generate linear programs
that take all the constraints of the original proof into account.

\begin{figure}
\begin{center}
\includegraphics[width=10cm]{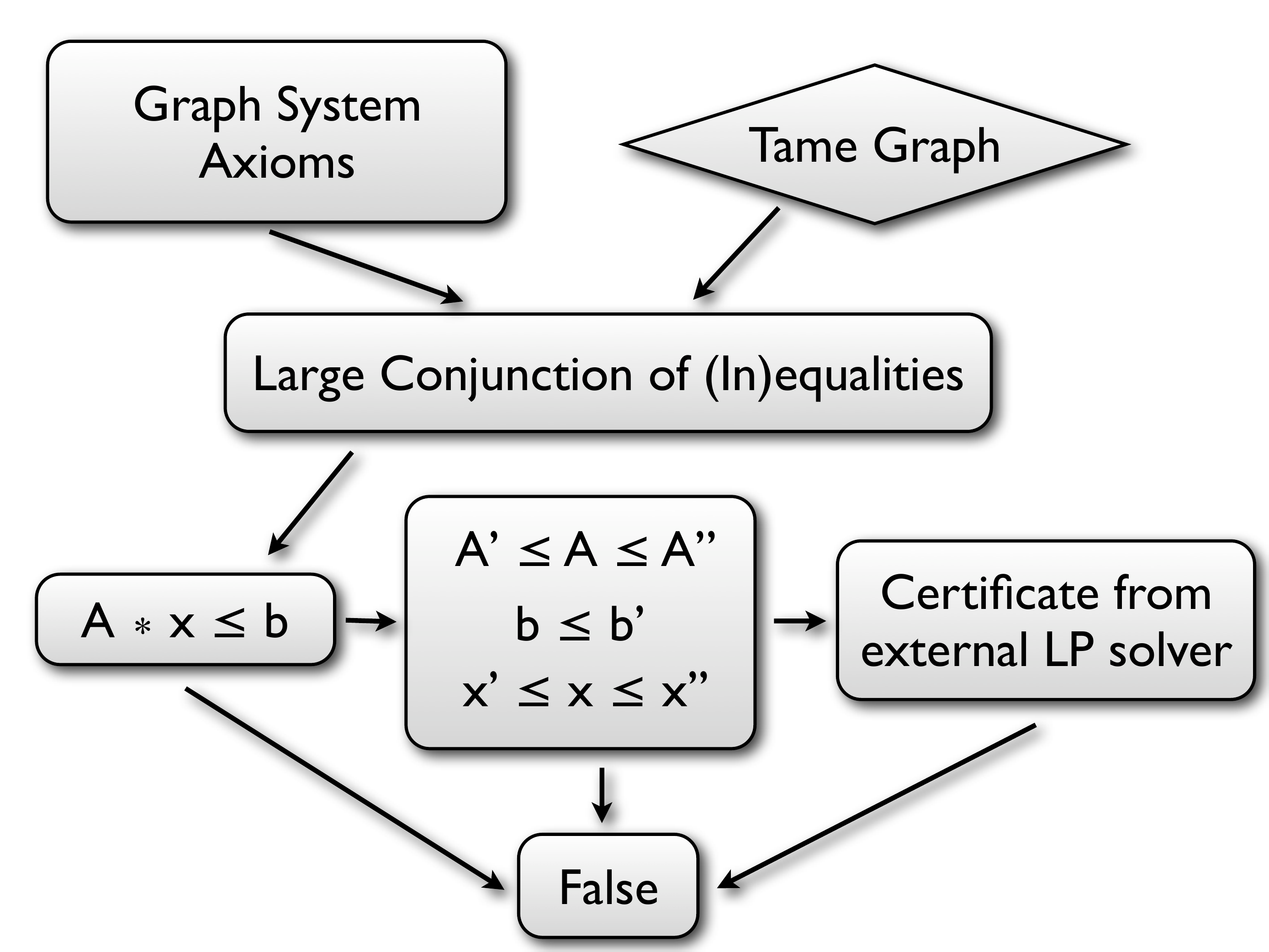}
\end{center}
\caption{Refuting a potential counterexample to the Kepler conjecture}
\label{fig:lpapproach}
\end{figure}

\part{Addendum to and Errata in the Original Proof}

\section{Biconnected graphs}
\label{sec:biconnected}

This section gives further detail to the argument of \cite[\S12.7~p.131]{Hales:2006:DCG:4}.  There it is claimed that
the proof of the main estimate~\cite[Theorem~12.1]{Hales:2006:DCG:4} can be reduced to the case of polygonal standard regions.
This claim is correct. 
However, the justification of this claim is not complete in the original proof.
This section gives complete justification of the claim.   The main result is Theorem~\ref{thm:biconnected}, stated below.

This section patches the original proof.  As such, it should be imagined this section to be inserted as an addendum, directly following Section~12.13 in the original proof, as a new Section~12.14.  (In fact, Section~12.13 is itself an addendum to the 1998 preprint, making this section the second addendum to Section~12.)  If the actual insertion were to be carried out, the corresponding changes to the numbering of the results in this section would be made: Theorem~\ref{thm:biconnected} below would be inserted as a new Theorem~{12.21}, and so forth. 

In the original proof, the boundaries of standard regions may fail to be
simple polygons.  In fact, the failure of the boundaries to be polygons may be quite
severe: the boundaries may contain multiple components and articulation vertices.
An {\it articulation}
vertex in a graph is a vertex whose removal increases the number of connected
components of the graph.  A connected graph is {\it biconnected} if it contains no
articulation vertex.

In the original proof, each standard region is first prepared by a process
of erasure (described in \cite[\S11.1]{Hales:2006:DCG:4}), then deformed in a way
to transform its boundary into a simple polygon.  The deformation
is made in such a way that the values of 
two key functions $\op{vor}_{0,R}$ and $\tau_{0,R}$ are left unchanged.
These are the functions that enter into the main estimate~\cite[Theorem~12.1]{Hales:2006:DCG:4}.

What the original proof fails to consider
is a particular hypothetical situation where the deformation might fail.  This situation is illustrated in Fig.~\ref{fig:biconnected}, where the rigid movement of set of vertices (such as the illustrated triangle with vertex $w$)
is blocked by a nearby vertex $v$ that is not visible from $w$, when the distance between $v$ and $w$ drops
to the minimum value $2$. This section presents a proof that this hypothetical situation does
not occur.

 The strategy of the proof in this section is to use a deformation argument.  A general decomposition star is deformed until the graph attached to it becomes biconnected.  In a biconnected planar graph with at least three vertices, each face is a simple polygon.  Thus, by producing a biconnected graph, we achieve our objective.  
In the original proof of the Kepler conjecture, biconnected graphs are not
mentioned by name.   Nevertheless, most of the graphs that occur in the late stages
of the original proof are biconnected.

In more detail, the proof in this section will produce a sequence of
admissible deformations of a decomposition star $D$ in such a way that
the individual standard regions $R$ are preserved in number and identity (but
not in shape) by the deformations.
The deformations will preserve the values of $\op{vor}_{0,R}(D)$ and
$\tau_{0,R}(D)$, for each standard region $R$. 
The deformations will change the combinatorial structure of the graph
formed by all the subregions of the decomposition star. 
At the conclusion of the deformation, the decomposition star $D$ will have
a graph (formed by all the subregions of the decomposition star) 
that is biconnected.  As in the original proof, the decomposition star
is first simplified by a process of erasure, before the deformations begin.

In the remainder of this section, 
we adopt notation, definitions, and conventions without further comment
from \cite{Hales:2006:DCG:4}.  

\begin{figure}
\begin{center}
\includegraphics[width=5cm]{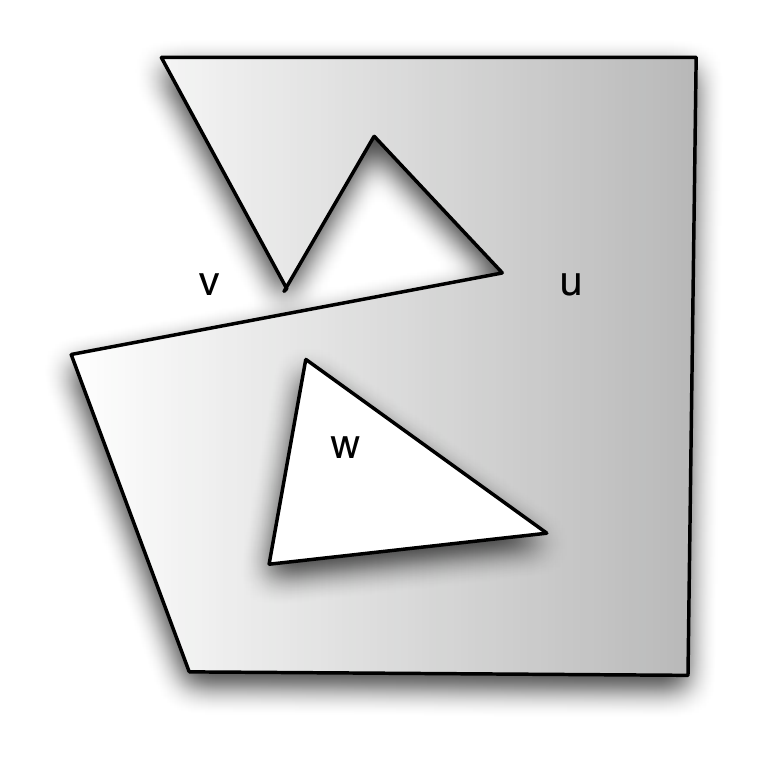}
\end{center}
\caption{In this figure (not to scale), segments represent geodesic arcs on the unit sphere.  The shaded region is a subregion whose boundary is not a simple polygon.
We wish to rigidly slide the triangle containing the vertex $w$ until a new visible distinguished edge forms
(say from $w$ to $u$).  We will show that this rigid motion does not decrease the distance between
two vertices (say $w$ and $v$)  to the minimum distance $2$.}
\label{fig:biconnected}
\end{figure}

\subsection{Context}

We work in the following restrictive
context for Theorem~\ref{thm:biconnected} and its proof.  
We fix a packing centered at a vertex at the origin.
As described in Section~12.6 of the original proof, we assume that all upright quarters
are erased, except loops (that is, those surrounded by anchored simplices).    
Let $U$ be the set of (non null) vertices of of height at most $2.51$.  As usual, we say two edges $\{u_1,u_2\}$ and $\{u_1',u_2'\}$ {\it cross}
if the interiors of the triangles formed by $\{0,u_1,u_2\}$ and $\{0,u_1',u_2'\}$ intersect.

We form the set of edges $E'$ between vertices in $U$, consisting of
\begin{itemize}
\item all standard edges; that is, $\{v,w\}\subset U$ such that $0<\|v-w\|\le 2.51$.
\item all edges $\{v,w\}\subset U$ of an anchored simplex, whenever the upright diagonal of the anchored simplex is an unerased loop.
\item all edges $\{v,w\}\subset U$ such that $0<\|v-w\|\le\sqrt8$, where $\{v,w\}$ does not cross any other edge in previous two items.  (If two of these edges cross, pick only one of them. This can only happen with conflicting diagonals
of a quad cluster.)
\end{itemize}
These edges do not cross.  A {\it special simplex} $\{0,u,v,w\}$ has one edge $\{v,w\}$ of length at least $\sqrt8$,
called the {\it special edge}.  
A special edge has length at most $3.2$.
The other vertex $u$ is called a {\it special vertex} (or {\it special corner}).
Let $E=E'\setminus S$, where $S$ is the set of special edges (that
is, the edges of special simplices shared with an anchored simplex).
 The projection of the line segments formed by $E$ to the
unit sphere is a planar graph.
The complement of this graph in the unit sphere
is a disjoint union of connected components.  The closures of these connected components are called {\it subregions}.

We call a {\it loop subregion} one that contains an unerased loop $\{0,v\}$.
If $R$ is a loop subregion, then there are no enclosed vertices of height $\le 2.51$ over the
subregion.  The corners of $R$ are the anchors of the upright diagonal together with the special corners
of the subregion.  The subregion $R$ is star convex with center point at the projection of $v$ to the unit sphere.
It follows that the boundary of $R$ is a simple polygon.

The graph $\Gamma$ with vertices $U$ and edges $E$ is not necessarily connected.  The aim is to deform
$U$ (and $D$) to create a biconnected graph (without changing the values of $\op{vor}_{0,R}(D)$ and $\tau_{0,R}(D)$ for {\it standard regions} $R$; the values for individual subregions will change).  Once the graph $\Gamma$
is biconnected, the subregions are simple polygons as desired.

The deformation of $U$ is {\it admissible} if it satisfies the following three conditions.
\begin{itemize}
\item If the graph $\Gamma$ is not connected,
the deformation acts by a rotation about the origin 
on the vertices of a single chosen connected component of $\Gamma$,
leaving all other vertices of $U$ fixed.  (For example, in Figure~\ref{fig:biconnected}, the entire triangle with vertex $w$ may be rotated.)  In particular,
$\|v\|$, for $v\in U$, is constant.
\item If the graph $\Gamma$ is connected, but not biconnected, with a chosen articulation vertex $a$, then
the deformation acts by a rotation about the axis $\{0,a\}$ on the vertices of a single component of $\Gamma\setminus\{a\}$,
leaving all other vertices of $U$ fixed.
\item The distance between $v,w$ remains at least $2$, for all $v,w\in U$.
\end{itemize}

The deformation stops when either of the following {\it halting conditions} are met:
{
\renewcommand{\labelenumi}{(H{\theenumi})}
\begin{enumerate}
\item $\|v-w\|\le\sqrt8$, where the edge $\{v,w\}\not\in E'$, with $\{v,w\}\subset U$, does not cross any edge in $E'$, or
\item $\|v-w\|$ decreases to $2$ for some $v,w\in U$, and $\{v,w\}$ crosses an edge in $E'$.
\end{enumerate}
}

We warn that the deformation is allowed to assume configurations in which the distance between some $u\in U$ and the upper endpoint of an unerased upright diagonal $d$ is less than $2$.  This is not a problem because the vertex $u$ does not end up as a corner of the loop subregion to which the upright diagonal $d$ belongs.  (In particular, the halting condition (H1) prevents a new special vertex from being added to a loop subregion.) Thus, when a single subregion is considered in isolation, all distances between pairs of vertices in that subregion are at least $2$.

In the restrictive context that has been described in this section, we have the following result.

\begin{thm}\label{thm:biconnected}  
If the graph $\Gamma=(U,E)$ is not biconnected, then a nontrivial admissible
deformation 
of $U$ exists. 
The  halting condition (H2) never holds.  The admissible
deformation can always be continued until the halting
condition (H1) holds.
\end{thm}

\begin{figure}
\begin{center}
\includegraphics[width=10cm]{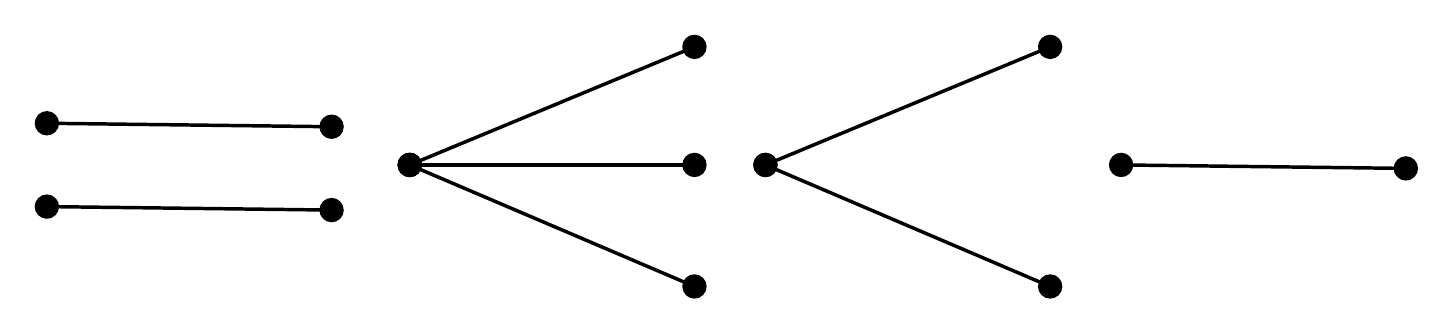}
\end{center}
\caption{Theorem~\ref{thm:biconnected} breaks into various cases: (a) two edges with
distinct endpoints, (b) three or more edges, (c) two edges, or (d) a single edge
passing through the triangle $\{0,v,w\}$.}
\label{fig:cases}
\end{figure}

We break the proof of Theorem~\ref{thm:biconnected} into cases (a), (b), (c), (d) according
to the number and combinatorial structure of the edges $\{u_1,u_2\}$
that pass through the triangle $\{0,v,w\}$.  See Figure~\ref{fig:cases}.
Lemma~\ref{lemma:double-edge} below gives Case (a): we cannot have two such edges $\{u_1,u_2\}$ and $\{u_1',u_2'\}$ with distinct endpoints.   It follows that there is an endpoint $u_2$ shared by every edge that passes through the
triangle $\{0,v,w\}$.  Lemma~\ref{lemma:three-edge} below gives Case (b): there cannot be three or more edges.  Finally, Cases (c) and (d) of one or two edges $\{u_1,u_2\}$ are treated in further subsections.  Each subsection is organized
around one of the main cases (with introductory Section~\ref{sec:pre}): Cases (a), (b), (c), and (d) are
treated in the four Sections~\ref{sec:22}, \ref{sec:13}, \ref{sec:12}, \ref{sec:11}, respectively.

When $U$ is deformed until the halting condition (H1) holds, a new edge $\{v,w\}$ can be added to $E$.
The theorem is repeated to add further edges, until a biconnected graph is obtained.  
The remainder of this section explains why the halting condition (H2) never holds.  For
this, we assume on the contrary, that $\|v-w\|=2$, where $\{v,w\}$ belong to different connected components
of $\Gamma$ (if $\Gamma$ is not connected) and different connected components of $\Gamma\setminus\{a\}$ (if $a$
is an articulation vertex of a connected graph $\Gamma$).  When this situation occurs, we say that $v,w$ belong to
different bicomponents.  Without loss of generality, we may also assume that $\{v,w\}$ crosses some edge of $E'$.

\subsection{Preliminary lemmas}\label{sec:pre}

Lemma~\ref{lemma:1} below shows that a vertex does not cross over an edge during a deformation; the halting conditions for the vertex are met before it reaches the edge.  Recall that the term {\it geometric considerations} refers to a specific collection of methods introduced in \cite[\S4.2]{Hales:2006:DCG} to prove the existence and non-existence of various simple configurations of points in $\ring{R}^3$.

\begin{lemma}\label{lemma:1} Let $S=\{0,v,,u_1,u_2\}$ be a set of four distinct points in $\ring{R}^3$ whose pairwise distances are at least $2$.  Suppose that $\|u\|\le 2.51$, for all $u\in S$.  If the segment $\{0,v\}$ meets the segment $\{u_1,u_2\}$, then $\|v-u_1\|,\|v-u_2\|\le 2.51$.
\end{lemma}

\begin{proof} This follows by geometric considerations.
\end{proof}

The next two lemmas give conditions under which the halting condition (H2) cannot hold.

\begin{lemma}\label{lemma:prelim}  Let $S=\{v,w,u_1,u_2,v_0\}$ be a set of five distinct points in $\ring{R}^3$ whose
pairwise distances are at least $2$. Suppose that the segment $\{v,w\}$ passes through
$\{v_0,u_1,u_2\}$. Assume
$$
\begin{array}{lll}
\|u_1-u_2\|\le 3.2,\\
\|v_0-u\|\le 2.51,\text{ for } u \in S.\\
\end{array}
$$
Then $\|v-w\|>2$.
\end{lemma}

\begin{proof} This follows by geometric considerations.
\end{proof}

\begin{lemma}\label{lemma:291}  Let $S=\{v,w,u_1,u_2,v_0\}$ be a set of five distinct points in $\ring{R}^3$ whose
pairwise distances are at least $2$. Suppose that the segment $\{u_1,u_2\}$ passes through
$\{v_0,v,w\}$. Assume
$$
\begin{array}{lll}
\|u_1-u_2\|\le 2.91,\\
\|v_0-u\|\le 2.51,\text{ for } u \in S.\\
\end{array}
$$
Then $\|v-w\|>2$.
\end{lemma}

\begin{proof} This follows by geometric considerations.
\end{proof}

By Lemmas~\ref{lemma:prelim} and \ref{lemma:291}, we may assume that for each edge $\{u_1,u_2\}$ that $\{v,w\}$ crosses,
we have that $\{u_1,u_2\}$ passes through $\{0,v,w\}$ and that $\|u_1-u_2\|>2.91$.  This means that the edge $\{u_1,u_2\}$ is an edge of an anchored simplex,
so that the edge is special or bounds a loop subregion.
This loop subregion will provide the key to the proof, because we will show
that it prevents the distance between $v$ and $w$ from becoming $2$,
as assumed  (Fig.~\ref{fig:biconnectedloop}).

\begin{remark}
As an aside, we mention that this issue of biconnected graphs is a major issue only in the 
proof of the Kepler conjecture and not
in the proof of the dodecahedral conjecture~\cite{Hales:2008:Dodec}.  In the proof of the dodecahedral conjecture there are no loops of anchored simplices, and without loops there are no difficulties: edges of length at most
$\sqrt8$ can be be used instead of the set $E'$.  Lemmas~\ref{lemma:prelim} and~\ref{lemma:291} suffice.
\end{remark}

\begin{figure}
\begin{center}
\includegraphics[width=5cm]{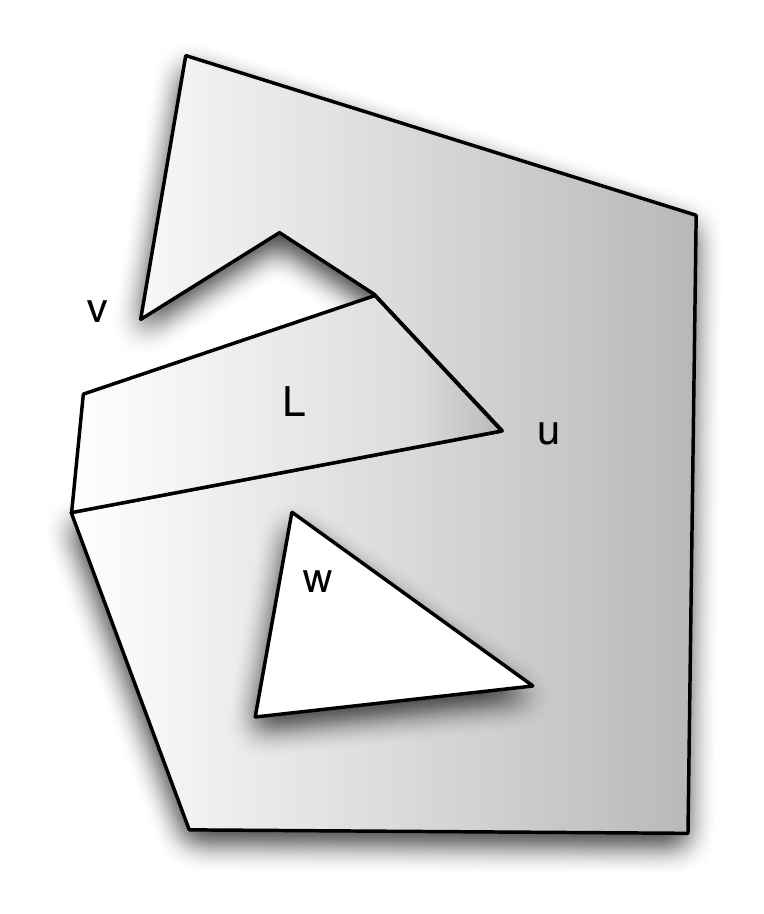}
\end{center}
\caption{In this figure (not to scale), segments represent geodesic arcs on the unit sphere.  The two shaded regions represent
subregions.    There are possibly other subregions partitioning the unshaded portions of the figure.  These other subregions are not represented in this figure.  The shaded region marked $L$ represents a loop subregion. The other shaded region represents a subregion with two boundary components that is undergoing a deformation, by the motion of the triangle with vertex $v$.  The situation of Fig.~\ref{fig:biconnected} does not exist
(as will be shown).  The  intervening loop subregion $L$  forces
$v$ and $w$ to be more than the minimum distance apart.}
\label{fig:biconnectedloop}
\end{figure}

\subsection{Case (a): Edge crossings with distinct endpoints}\label{sec:22}

\begin{lemma}\label{lemma:double-edge} 
There does not exist a set $S=\{0,v,w,u_1,u_2,u_1',u_2'\}$ of seven distinct points
in $\ring{R}^3$ whose pairwise distances are at least $2$ and that satisfies the following conditions.
\begin{itemize}
\item The edges $\{u_1,u_2\}$ and $\{u_1',u_2'\}$ do not cross.
\item The edges $\{u_1,u_2\}$ and $\{u_1',u_2'\}$ both pass through $\{0,v,w\}$.
\item $\|u\|\le 2.51$ for all $u\in S$.
\item $\|v-w\|=2$.
\item $\|u_1-u_2\|,\|u_1'-u_2'\|\le 3.2$.
\item $u_1$ and $u_1'$ lie in the same half-space bounded by the plane $\{0,v,w\}$.
\item The directed segment from $v$ to $w$ crosses the segment $\{u_1',u_2'\}$ before the segment $\{u_1,u_2\}$.
\end{itemize}
\end{lemma}

Note that the last two conditions can always 
be achieved by suitable labels on the points $\{u_1,u_2,u_1',u_2'\}$.

\begin{proof}
This will follow as a direct consequence of Lemmas~\ref{lemma:7small} and~\ref{lemma:7big} below.
\end{proof}

\begin{lemma}\label{lemma:diam} 
Let $S=\{u,v,w\}$ be a triangle such that each side has length at least $2$, and such that $\|u-v\|\le 2.51$, $\|u-w\|\le 2.51$, $\|v-w\|=2$.  Let $X$ be the set of points in the convex hull of $S$ that have distance at least $1.2$ from each vertex of $S$.  Then the diameter of $X$ is less than $1.044$.
\end{lemma}

\begin{proof} Let $x=\|u-v\|$ and $y=\|u-w\|$.  Assume $x\ge y$.  As $u$ moves away from $v$, along a fixed line through $v$ and an initial position $u_0$, the region $X$ expands.  Thus, we may assume that $x=2.51$.   The boundary of $X$ is a polygonal curve consisting of line segments and concave arcs of circles.  The diameter is realized by the distance between two vertices $p_i(y)$ of the polygonal curve. We consider two cases according to $y\le 2.4$ and $y\ge 2.4$ because the structure of the polygonal curve changes at $y=2.4$.  We calculate $\|p_i(y)-p_j(y)\|$ directly, checking for each $(i,j)$ that the distances are less than $1.044$.
\end{proof}

\begin{lemma}\label{lemma:7small}  
Let $S=\{0,,v,w,u_1,u_2,u_1',u_2'\}\subset\ring{R}^3$
be a configuration of seven points that satisfies
the conditions of Lemma~\ref{lemma:double-edge}.  Then there exist
a point $p$ on the segment $\{u_1,u_2\}$ and  a point $p'$ on the segment
$\{u_1',u_2'\}$ such that $\|p-p'\|<1.044$.
\end{lemma}

\begin{proof} Let $u\in\{0,v,w\}$.  By the metric constraints, the
distance from $u$ to the segment $\{u_1,u_2\}$ (resp. $\{u_1',u_2'\}$) is at least
$$\sqrt{2^2 - (3.2/2)^2} = 1.2.$$
Let $p$ (resp. $p'$) be the point of intersection of the segment $\{u_1,u_2\}$ 
(resp. $\{u_1',u_2'\}$) with
the convex hull of $\{0,v,w\}$.  By Lemma~\ref{lemma:diam},
we have $\|p-p'\|<1.044$.
\end{proof}

\begin{lemma}\label{lemma:abc}
Let $S=\{u_1,u_1',u_2,u_2'\}\subset\ring{R}^3$ be a set of four distinct points such that the distance between each pair of points is at least $2$.  Assume that
$$
\|u_1-u_2\|\le 3.2,\quad \|u_1'-u_2'\|\le 3.2.
$$
If any of the following conditions hold:
{
\renewcommand{\labelenumi}{(\Alph{enumi})~~}
\begin{enumerate}
\item $\|u_1-u_2'\|,\|u_1'-u_2\|\ge 2.91$;
\item $\|u_1-u_1'\|,\|u_2-u_2'\|\ge 2.85,\quad \|u_1-u_2'\|\ge 2.91$;\quad or
\item $\|u_1-u_1'\|\ge 3.64,\quad \|u_1-u_2'\|\ge 2.91$;
\end{enumerate}
}
then every point on the segment $\{u_1,u_2\}$ has distance greater than $1.044$ from every point on the segment $\{u_1',u_2'\}$.
\end{lemma}

\begin{proof} Assume for a contradiction that the assumptions hold and that
the conclusion is false for some configuration.  The metric constraints can be used to show that the segments $\{u_1,u_2\}$ and $\{u_1',u_2'\}$ can be stretched along their axes without decreasing any edge length.  
Thus we may assume without loss of generality that $\|u_1-u_2\|=\|u_1'-u_2'\|=3.2$.  Decreasing one dihedral angle of the simplex $S$ at a time, we may move the segments closer together, until all four edges $\{u_i,u_j'\}$ attain their minimum length.  
Then we have three rigidly determined simplices (A), (B), (C) (with equality constraints).  An explicit coordinate calculation of the distance between the two segments shows that the distance is greater than $1.044$ in each case.
\end{proof}

\begin{lemma}\label{lemma:7big}  
Let $S=\{0,,v,w,u_1,u_2,u_1',u_2'\}\subset\ring{R}^3$
be a configuration of seven points that satisfies
the conditions of Lemma~\ref{lemma:double-edge}.  Then there do not exist
a point $p$ on the segment $\{u_1,u_2\}$ and  a point $p'$ on the segment
$\{u_1',u_2'\}$ such that $\|p-p'\|<1.044$.
\end{lemma}

\begin{proof}  Assume for a contradiction that such $p,p'$ exist.
In Lemma~\ref{lemma:abc}, we may assume that none of the
conditions A, B, C hold.  By obvious symmetry, without loss of generality, we may assume by case A
that $\|u_2-u_1'\|<2.91$.  By Lemma~\ref{lemma:291}, this implies that $\{u_1',u_2\}$ does not pass through $\{0,v,w\}$.  By the conditions of the lemma, the segment $\{u_1,u_1'\}$ does not meet the plane $\{0,v,w\}$, so that $\{u_1,u_1'\}$ does not pass through $\{0,v,w\}$.  This means that the triangle $\{u_1,u_1',u_2\}$ is linked around $\{0,v,w\}$; and  some edge of $\{0,v,w\}$ passes through $\{u_1,u_1',u_2\}$.  Up to symmetry there are two cases: (1) $\{0,v\}$ passes through $\{u_1,u_1',u_2\}$, or (2) $\{v,w\}$ passes through $\{u_1,u_1',u_2\}$.

In the first case, recall that $\{u_1',u_2'\}$ does not cross $\{u_1,u_2\}$.
So $u_2'$ is enclosed over $(0,\{u_1,u_1',u_2\})$.  Then $\{u_2',u_1\}$
and $\{u_2',u_1'\}$ pass through $\{0,v,w\}$.  By Lemma~\ref{lemma:291}, we
have $\|u_2'-u_1\|,\|u_2'-u_1'\|\ge 2.91$.  As we are assuming that C does not hold, we have $\|u_1-u_1'\|\le 3.64$.  We claim that $\{0,u_2'\}$
passes through $\{u_1,u_1',u_2\}$.  Otherwise, $u_2'$ lies in the convex hull of $\{0,u_1,u_1',u_2\}$ and a coordinate calculation shows that the upper bounds on the edges of the simplex $\{0,u_1,u_1',u_2\}$ are inconsistent with the lower bounds on the distances from $u_2'$ to the vertices of the simplex.  Since $\{0,u_2'\}$ passes through $\{u_1,u_1',u_2\}$ we may use geometric considerations to show that $\|u_2'\| > 2.51$.  This is contrary to the hypotheses of Lemma~\ref{lemma:double-edge}.

In the second case, $\{v,w\}$ passes through $\{u_1,u_1',u_2\}$ and through $\{u_1',u_2',u_2\}$.  Geometric considerations give $\|u_1-u_1'\|>2.85$ and $\|u_2-u_2'\|> 2.85$.  Assuming that B does not hold gives $\|u_1-u_2'\|< 2.91$.  The first paragraph of the proof now gives that $\{u_2,u_2',u_1\}$ links around the triangle $\{0,v,w\}$.  The edge $\{v,w\}$ does not pass through $\{u_2,u_2',u_1\}$.  (This can be seen by drawing the relative positions of $p(u)$, for $u\in S$, in the projection $p$ of the points to a plane orthogonal to $\{v,w\}$.)  Thus, we are in the first case, which has already been treated.
\end{proof}

\subsection{Case (b): Triple edge crossings}\label{sec:13}

By Lemma~\ref{lemma:double-edge}, there is a common endpoint $u_2$ such that every edge
of $E'$ that passes through $\{0,v,w\}$ has $u_2$ as an endpoint.
Next we show that there cannot be three such edges.

\begin{lemma}\label{lemma:three-edge}
There does not exist a set of seven distinct points
$$S=\{0,v,w,u_1,u_1',u_1'',u_2\}$$ in $\ring{R}^3$ that satisfies
the following conditions.
\begin{itemize}
\item The distance between each pair of distinct points in $S$ is at least $2$.
\item The edges $\{u_1,u_2\}$, $\{u_1',u_2\}$, and $\{u_1'',u_2\}$ pass through
$\{0,v,w\}$.
\item $\|u\|\le 2.51$, for all $u\in S$.
\item $\|v-w\|=2$.
\item $\|u-u_2\|\le 3.2$, for $u=u_1,u_1',u_1''$.
\end{itemize}
\end{lemma}

\begin{proof}
Assume for a contradiction that $S$ exists.
We may pivot $w$ around the axis $\{0,v\}$ until $\|w-u_2\|\le 2.51$.  (The metric constraints on
edge lengths show that the condition $\Delta>0$
is preserved for the simplices $\{u,u_2,v,w\}$ and $\{u,u_2,w,0\}$, for $u=u_1,u_1',u_1''$, throughout
this pivot.)  A similar
pivot of $v$ gives $\|v-u_2\|\le 2.51$.
We may order the vertices in cyclic order around $\{0,u_2\}$ as
$$(w_1,w_2,w_3,w_4,w_5)=(w,u_1,u_1',u_1'',v),$$ 
so that setting
$d(w_i,w_j) = \dih(0,u_2,w_i,w_j)$, we have
$$d(w_1,w_5)=\sum_{i=1}^4 d(w_i,w_{i+1})\ge d(w_2,w_3)+d(w_3,w_4).$$
Interval calculations\footnote{\calc{2799256461}, \calc{5470795818}} 
give $d(w_2,w_3),d(w_3,w_4)\ge 0.7$
and $d(w_1,w_5)< 1.4$. We obtain an immediate contradiction:
$$1.4 > d(w_1,w_5) \ge 0.7 + 0.7.$$
\end{proof}

\subsection{Case (c): Double edge crossings}\label{sec:12}

This subsection treats the case of two edges crossings in the proof
of Theorem~\ref{thm:biconnected}.
We continue to assume the general context of Theorem~\ref{thm:biconnected}.  As usual, 
the edge $\{u_1,u_2\}\in E'$ crosses $\{v,w\}$.

\begin{lemma}\label{lemma:special}
Let $\{u_1,u_2,w,v\}$ be a set of four distinct points
in $\ring{R}^3$ (in the given context).  
Assume that
 $\|v-u_i\|\le 2.51$, for $i=1,2$.
Assume that no edge of $E'$ crosses $\{v,w\}$
in the open 
half-space $A$ containing $v$ bounded by the plane $\{0,u_1,u_2\}$.
(That is, $\{u_1,u_2\}$ is the first edge to cross $\{v,w\}$, moving
from $v$ toward $w$.)
Assume there is a loop subregion $L$ along 
$\{u_1,u_2\}$ on the $A$-side
of $\{u_1,u_2\}$.  Then $\{u_1,u_2\}$ is a special edge of $E'$ with
corner $v$.
In particular, $\{u_1,u_2\}\not\in E$, so that it is not a bounding
edge of a subregion.
\end{lemma}

\begin{proof}
Assume for a contradiction that $\{u_1,u_2\}$ is not special.
Note that loop subregions have simple polygonal boundaries and
remain rigid under all the deformations.  In particular,
the upright diagonal, special corners, and so forth
remain rigidly positioned with respect to the corners of the subregion.

Since there are no further edges crossing $\{v,w\}$, the subregion
$L$ extends to $v$.  Hence $v$ is a corner of the subregion $L$.
It is either an anchor or a special corner (with respect to $L$). However, it cannot be a special
corner, by the assumption that $\{u_1,u_2\}$ is not a special edge.
Hence it is an anchor.  Also, $u_1$ and $u_2$ are anchors.

To reach a contradiction, 
we consider possible locations of the upright diagonal $\{0,u\}$,
and show that it has nowhere to go (Figure~\ref{fig:nogo}).  
Since $\{u_1,u_2\}$ is an edge,
the points $u_1,u_2$ are consecutive anchors. This prevents $u$ from lying over the region $B$.  
Also, the upright diagonal of any
unerased loop has at least four anchors (say $v,u_1,u_2,w$).  Moreover, some anchored simplex around the upright diagonal is not an upright quarter (because of the edge $\{u_1,u_2\}$).  Geometric considerations based
on these constraints (say $\|w-u_1\|\ge 2.51$) show that
the fourth anchor $w$ is not in $A$.  This prevents $u$ from
being located over the region $A$.
A vertex $u_j$ cannot be enclosed over an upright quarter $\{0,u,v,u_i\}$.  This excludes
 region $C$.  Finally, an edge $\{v,u_i\}$
of length at most $2.51$ cannot pass through a triangle $\{0,u,u_j\}$
of sides at most $2.51,2.51,\sqrt8$ (excluding $D$).
\end{proof}

\begin{figure}
\begin{center}
\includegraphics[width=5cm]{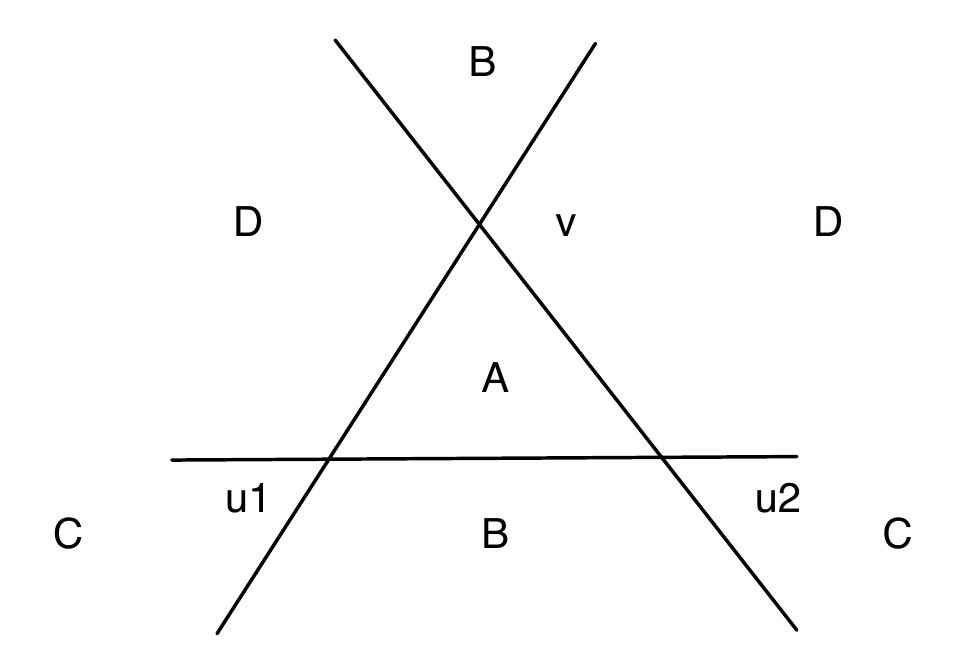}
\end{center}
\caption{The upright diagonal $u$ cannot be placed over any of
the regions $A,B,C,D$.  The lines (not to scale) represent geodesic arcs on the
sphere passing through the pairs of points in $\{p(u_1),p(u_2),p(v)\}$,
where $p$ denotes projection to the unit sphere.}
\label{fig:nogo}
\end{figure}

\begin{lemma}\label{lemma:circuit} 
Let $\{u_1,u_2,w,v\}$ be a set of four distinct points
in $\ring{R}^3$ (in the given context).  
Assume that
 $\|v-u_i\|\le 2.51$, for $i=1,2$.
Assume that no edge of $E$ crosses $\{v,w\}$
in the open 
half-space $A$ containing $v$ bounded by the plane $\{0,u_1,u_2\}$.
(That is, $\{u_1,u_2\}$ is the first edge to cross $\{v,w\}$, moving
from $v$ toward $w$.)
Then 
both edges $\{u_1,v\}$ and $\{u_2,v\}$ belong to $E$.
In particular, there is a circuit of the graph $\Gamma$ through $v,u_1,u_2$.
\end{lemma}

\begin{proof} If
there is a loop subregion $L$ along 
$\{u_1,u_2\}$ on the $A$-side
of $\{u_1,u_2\}$, Lemma~\ref{lemma:special} implies that
$\{u_1,u_2\}$ is a special edge of $E'$ with
corner $v$.  In particular, $\{u_1,v\}$ and $\{u_2,v\}$ are
edges of $E$. The conclusion follows in this case.

 Now assume that there is no loop subregion along
$\{u_1,u_2\}$ on the $A$ side of $\{u_1,u_2\}$.

Let $S$ be the finite set of points of $U$ enclosed
over the simplex $\{0,u_1,u_2,v\}$.  We show by contradiction
that $S$ is empty.
The plane $\{0,v,w\}$ separates
$S$ into a disjoint union $S = S_1\cup S_2$, according to those
in the same half-space as $u_i$, $i=1,2$.
We form the convex hull of the projection $p$ to the unit sphere of the 
points $S_i\cup\{u_i,v\}$.  As in~\cite[\S12.13]{Hales:2006:DCG:4},
form a sequence of geodesic arcs on the unit sphere from $p(u_i)$
to $p(v)$.  Let $p(w_i)$, for $w_i\in S_i$, be the other endpoint
of the arc starting at $p(u_i)$ (or set $w_i=v$ if $S_i=\emptyset$.  
For some $w'\in\{w_1,w_2\}\setminus\{v\}$,  the edges $\{w',u_i\}$ do not cross any edges
of $E$.  Furthermore, geometric considerations show that
$\|w'-u_i\|\le 2.51$, for $i=1,2$.  By the criteria for forming
edges of $E$, we must have $\{w',u_i\}\in E$ for $i=1,2$.  This
contradicts the assumption that $\{v,w\}$ does not cross any edges
of $E$ over $A$.  Hence $S=\emptyset$.

Since $S=\emptyset$, the edges $\{v,u_i\}$ do not cross any edges
of $E$.  By the criteria for forming edges of $E$, they belong to $E$.
This completes the proof.
\end{proof}

We are ready to prove the next major case of Theorem~\ref{thm:biconnected}.
We continue to work in the general context of that theorem, with $v$ and $w$
in different bicomponents of the graph $\Gamma$.

\begin{lemma}\label{lemma:double-cross}  
In this context, there does not exist a set of six points $\{0,v,w,u_1,u_1',u_2\}$
where $\{u_1,u_2\}$ and $\{u_1',u_2\}$ pass through $\{0,v,w\}$ and such that
  $$\|u_1-u_2\|\le 3.2,\quad \|u_1'-u_2\|\le 3.2.$$
\end{lemma}

\begin{proof}  We argue by contradiction.  We may assume that $\{u_1,u_1'\}$ are ordered so that the cyclic order around $\{0,u_2\}$ is
$(w_1,w_2,w_3,w_4)=(v,u_1',u_1,w)$. 
By the previous results, there are at most two edges that pass through $\{0,v,w\}$
in this manner.  In particular, the part of the line segment $\{v,w\}$ between the crossings
of $\{u_1,u_2\}$ and $\{u_1',u_2\}$ lies in a single subregion $L$.  

We claim that we do not have $\|v-u_1'\|,\|v-u_2\|,\|w-u_1\|,\|w-u_2\|\le 2.51$.  Otherwise, we break
into two cases and derive a contradiction as follows.  Either (1) $L$ is a loop subregion, or (2)
$L$ is not a loop subregion, but both regions adjacent to $L$ (along edges $\{u_1,u_2\}$, $\{u_1',u_2\}$)
are loop subregions.  In case (1),
$\{u_1,u_2,u_1'\}$ are corners of the loop subregion $L$.  Hence, they lie on a circuit in $\Gamma$
formed by the corners of that loop subregion.  
If $\|v-u_2\|,\|v-u_1'\|\le 2.51$, then by Lemma~\ref{lemma:circuit}, the points
$\{v,u_1',u_2\}$ lie on a circuit in $\Gamma$.  A similar conclusion holds if corresponding inequalities
hold $\|w-u_2\|,\|w-u_1\|\le 2.51$.  If all four inequalities hold, then these circuits put $v,w$ in the
same bicomponent of $\Gamma$, which is contrary to hypothesis.  
In case (2), then by Lemma~\ref{lemma:circuit}, $\{v,u_2\},\{v,u_1'\}\in E$, so $\{u_1',u_2\}$ is a special
edge and $L$ is a loop subregion.  This is contrary to the assumption of case (2).
Hence we may assume by symmetry and without
loss of generality that $\|v-u_1'\|\ge 2.51$ or $\|v-u_2\|\ge 2.51$.

We may stretch along the edges $\{u_2,u_1\}$, $\{u_2,u_1'\}$, moving $u_1,u_1'$, 
until $\|u_2-u_1\|=\|u_2-u_1'\|=3.2$.   We may add inequality
$$
\|u_2\|\le 2.23,
$$
for otherwise by geometric considerations $\|u_2-u_1'\|> 3.2$.  Similarly, $\|u_1'\|\le 2.23$.
If $\|u_2-v\|\ge 2.51$, we may pivot $v$ toward $u_2$ around the axis $\{0,w\}$ until $\|u_2-v\|\le 2.51$.
Similarly, we may assume that $\|u_2-w\|\le 2.51$.
Set $d(i,j) = \op{dih}(0,u_2,w_i,w_j)$.
Then interval arithmetic calculations\footnote{\calc{7431506800}, \calc{5568465464}, \calc{4741571261}, \calc{6915275259} } give the contradiction:
$$
1.3 > d(1,4) = d(1,2) +d(2,3)+d (3,4) > 0.5 + 0.8 + 0 = 1.3.
$$
\end{proof}

\subsection{Case (d): Single edge crossings}\label{sec:11}

This subsection treats the proof of Theorem~\ref{thm:biconnected}
in the case of a single edge crossing $\{u_1,u_2\}$.  This is
the final case of the proof.
We continue to assume the notation and general context of that theorem.
In particular,  $v$ and $w$ lie in different bicomponents of
the graph $\Gamma$.

\begin{lemma}  Let $\{0,u_1,u_2,v,w\}$ be a set of five distinct
points such that $\{u_1,u_2\}$ is the only edge of $E'$ that
crosses $\{v,w\}$.  Then $\|v-w\|>2$.
\end{lemma}

\begin{proof} We assume for a contradiction that $\|v-w\|=2$.
We consider four cases depending on lengths.  

{\it Case 1:} $\|u-u_i\|\le 2.51$, for $i=1,2$ and $u=v,w$.
By Lemma~\ref{lemma:circuit}, there are circuits running through
$\{u,u_1,u_2\}$, for $u=v,w$.  This is contrary to the assumption
that $v,w$ lie in different bicomponents of the graph $\Gamma$.  (In the remaining cases, there is
no loss in generality to assume $\|w-u_2\|\ge 2.51$.)

{\it Case 2:} $\|w-u_2\|\ge 2.51$, $\|v-u_1\|\ge 2.51$.
Geometric considerations give the contradiction
$\|u_1-u_2\|>3.2$.

{\it Case 3:} $\|w-u_2\|\ge 2.51$, $\|v-u_2\|\ge 2.51$.
Geometric considerations gives the contradiction
$\|u_1-u_2\|>3.2$.

{\it Case 4:} $\|w-u_2\|\ge 2.51$, $\|v-u_i\|\le 2.51$, for $i=1,2$.
The edge $\{u_1,u_2\}$ cannot be a special edge of $E'$.  Otherwise,
$v,w$ are corners of the same loop subregion.  This contradicts
the running assumption that these two vertices belong to separate
bicomponents of the graph $\Gamma$.  By Lemma~\ref{lemma:special}, there is no loop subregion along $\{u_1,u_2\}$ on the $v$-side.
Since $\{u_1,u_2\}$ has length greater than $\sqrt8$, there
is a loop subregion $L$ bounded by the edge, and it must then lie
on the $w$-side.  Thus, $w$ is a corner of $L$ and the circuit of
$\Gamma$ described by the boundary of $L$ passes through $w,u_1,u_2$.
By Lemma~\ref{lemma:circuit}, there is a circuit of $\Gamma$ through
$v,u_1,u_2$.  Hence, $v,w$ lie in the same biconnected component,
which is contrary to the running assumption.
\end{proof}

\section{Errata listing}

The abridged version of the Kepler conjecture
in the Annals \cite{Hales:2005:Annals}
was generated by the same tex
files as the unabridged version in \cite{Hales:2006:DCG}.
For this reason,
it seems that every correction to
the abridged version should also be a correction to the unabridged version.
We list errata in the
unabridged version. The same list applies to corresponding 
passages in the abridged version.

Each correction gives its location in \cite{Hales:2006:DCG}.
The location
\line+n counts down from the top of the page, or
if a section or lemma number is provided, it
counts from the top of that organizational unit.
The location \line-n counts up from the bottom
of the page. Footnotes are not included in the
count from the bottom.  Every line containing
text of any sort is included in the count,
including displayed equations, section headings,
and so forth.  The material to the left of $\lto$ 
indicates original text, and material to the right of the
arrow gives replacement text.  
The original text and replacement text appear in italic.
Comments about the corrections appear in roman. 

In addition to the corrections to the text mentioned below, 
there have been some corrections to the computer code, including some typos in the
listings of nonlinear inequalities.
They are described in detail in~\cite{Hales:2008:Errata}.

\subsection{Listing}\hfill\break
\parskip=0.2\baselineskip

\begin{\sz}
\baselineskip=1.2\baselineskip

[p.47,Lemma~5.16] $Q\lto F$

[p.49,\line+2] {\it supposed} \lto {\it suppose}
	
[p.63,Lemma~7.10]
	${\mathcal S}$-{\it system} \lto $Q$-{\it system}

[p.75,Remark~8.11]
	{\it show}\lto {\it shows}

[p.78,\line-7] {\it constraints} \lto{\it  constraint}

[p.86,\line+14] {\it Let $\{0,v\}$ be 
          the diagonal of an upright quarter in the $Q_0$
        \lto
       Let $v$ be a vertex with $2t_0<|v|<\sqrt8$.}
           {\rmx Section~9 assumes that the diagonal belongs to
          a quarter in the $Q$-system, but Lemma~10.14 uses these
          results when $\{0,v\}$ has $0$ or $1$ anchors.  To make
          this coherent, we should assume throughout Section~9 that
          we have the weaker condition that whenever $\{0,v\}$ has
          two or more anchors, it is a diagonal of a quarter in the $Q$-system.
          The proofs of Section~9 all go through in this context.
           (Lemma~9.7 is all that is relevant here.)}

[p.87,Definition~9.3]
	{\rmx In definition of $\Delta(v,W^e)$, we
	can have some $Q$ (as in Fig~9.1)
	with negative orientation.
	In this case, $E_v\cap E_i$ can clip
	the other side.  We want the object
	without clipping.   $\Delta(v,W^e)$ should be understood as the
        unclipped object.}
	
[p.88,Definition~9.6]
	{\rmx The definition is poorly worded.  First of
	all, it requires that the subscript to
	$\epsilon$ to be a vertex, but then in
	the displayed equation, it makes $w/2$ the
	subscript, which is not a vertex.   To
	define $\epsilon'$, move from $w/2$ along
	the ray through $x'$ until an edge of the
	Voronoi cell is encountered.  If $v,w,u$
	are the three vertices defining that edge,
	then set $\epsilon'_v(\Lambda,x)=u$.
	Degenerate cases, such as when two different
	edges are encountered at the same time,
	can be resolved in any consistent fashion.  In~\cite{hales:2008:collection},
      these degeneracies are avoided altogether, by replacing functions 
      $\epsilon,\epsilon'$ with sets $\Phi,\Phi'$.
     }
	
[p.88,Lemma~9.7,\line+2] 
	$w$ {\it  and}  $v$\lto $w$ {\it and}  $u$

[p.88,L.~9.7,Claim~1]
	\text{{\it  with }} $|w - w'|\le 2t_0$, \text{ {\it and} }
	\lto \text{ {\it with} }

[p.88,L.~9.7,\line+5]
         {\it Then: $\lto$ Let
          $
          R'_w = \{x\in R_w \cap(0,\{u,w\})\mid 
          \epsilon_0(x,\{u,w\}) = u\}.
          $
          Assume that $R'_w$ is not empty. Then:}

[p88,L.~9.7,Claim~3]
        $R_w \lto R'_w$

[p.89,\line+2]
	$
	\{w,v\}\lto\{w,u\}
	$

[p.92,\line+16,\line+21]
   $     \max_j u_j \lto \max_j |u_j|$
	
[p.93,\line-4]
	$
	\text{{\it obstructed from} }w \lto
	\text{{\it obstructed from} }w'
	$
		
[p.93,\line-2]
	$
	\text{{\it from some}} \lto \text{{\it for some}}
	$

[p.99,\line+1]
        $
        \text{{\it start}} \lto \text{{\it star}}
        $

[p.105,Lemma~10.14]  {\rmx In the proof of the cases involving
   $0$ or $1$ anchor, a combination of the decompositions from
   Section~8.4 and Section~9 are used.  These decompositions haven't
   been shown to be compatible.  
   Instead, it is better to combine
   $\Delta(v,W)$ with $t_0$-truncation on the rest of the quad-cluster.
   With a $t_0$ truncation, we no longer have the non-positivity results
   from Section~8.  (The quoins give a positive contribution.) 
   However, a routine calculation shows that
   the estimate on $\Delta(v,W)$ is sufficiently small so that we still
   obtain a constant less than $-1.04\,\op{pt}$.}

[p.116][p.121] {\rmx Definition~11.7 allows masked
flat in definition of $3$-unconfined.
Definition~11.24 requires no masked flats
in the same definition.  Use Definition~11.24 (no masked flats), rather than~11.7.  
Where masked flats occur,
treat them with Lemma~11.23, parts (1) and (2).}

[p.116,\line+1] 
	$
	\text{{\it Lemma}}~4.16 \lto \text{{\it Lemma}}~4.17
	$

[p.117,before Lemma~11.9]
	$
	\text{\it two others} \lto \text{\it three others}
	$
	
[p.117,Def~11.8]
    $
    y1 \lto y_1
    $

[p.119,Definition~11.5]  {\rmx By definition, we require a masked flat quarter to
be a strict quarter.}
	
[p.121] {\rmx See p.116.}

[p.121,\line-5]
	$
	0.2274 \lto 0.02274
	$

[p.123. flat case (2)]  {\rmx It is missing
isolated quarters cut from the side.
To fix this, in condition 2(f), }
	$
	\eta_{456}\ge\sqrt2$ \lto
	$\eta_{456}\ge\sqrt2$ {\it or } $\eta_{234}\ge\sqrt2$
	
[p.124,Lemma~11.27]  {\it The bound of $0$ is established in Theorem 8.4} \lto
{\it The bound of $0$ for upright quarters appears in Lemma~8.12.  The bound of $0$ for the other anchored simplex appears in Lemma~8.7 or 8.13, depending on the circumradius.}

[p.126] 
{\rmx Theorem~12.1 should include $\sigma_R(D)\le s_n$
with $s_3 = 1\,\op{pt}$ and $s_4=0$, and
$\tau_R(D) \ge t_3 = 0$.}

[p.131] {\rmx Section~\ref{sec:biconnected} gives
the deformation arguments that produce a biconnected graph.}

[p.139,Lemma~12.18,proof,\line+3] 
	$C_0(|v|,\pi) \lto
	C_0^u(|v|,\pi)
	$
	
[p.139,Lemma~12.18,proof,\line+6] 
	$
	\tau_0(C_0^u(2t_0,\pi))-\pi_{\max}\lto
	\tau_0(C_0^u(2.2,\pi))-\pi_{\max}
	$

[p.144,\line+11,\line+17]
	$2t_0^2 \lto (2t_0)^2
	$

[p.146]
		$S_n^\pm$ \lto
	{\it of 3-crowded, 3-undefined, and
	4-crowded combinations}

[p.148,\S13.6]  {\rmx This entire
section is misplaced.  It belongs with
\S25.5 and \S25.6.}

[p.149,before 13.7]
{\it the diagrams\lto
	Figs~25.1--25.4}

[p.149,p.156] {\rmx The definition of $\delta_{loop}$ was accidently
dropped from the published version.  Set $\delta_{loop}(4,2)=0.12034 $
$\delta_{loop}(5,1)=0.24939$.  These constants and their properties
appear in the earlier 2002
arXiv preprint of the proof {\it The Kepler conjecture (Sphere Packings VI)}.
}

[p.156,Lemma~13.5,\line+4]
	{\it respectively for $\tau_R(D)$\lto
	respectively, for $\sigma_R(D)$ and }
	$\tau_R(D)$ 

[p.164,\line-1] 
	{\it This shows$\ldots$ occur.
	\lto This completes the proof.}

[p.173,\line+4] {\rmx Insert the subscript on $b$,
as in Proposition~15.5, starting on page 173:}
   $b$ \lto $b_q$.

[p.182,Lemma~16.7]  {\rmx The bound of $0$ has not been
shown to hold on each half.  This is not a 
direct consequence
of Theorem~8.4 as claimed.
This can be fixed as follows. 
Let $v_1$, $v_3$ be the corners giving the endpoints of the long edge of the
acute triangle at $0$, and let $v_2,v_4$ be the other two corners.
If either vertex $v_1,v_3$ 
has height greater than $2.3$ show that the $\op{vor}_0$-scored quad cluster scores
less than $-1.04\,\op{pt}$.  
For this, we may use the deformations of Lemmas~12.10.  The length of the diagonal along the acute face remains fixed and at least $\sqrt8$.  We claim these deformations produce a diagonal of length less than $\sqrt8$ between opposite corners of the quadrilateral.  (If not, the deformations produce a rhombus with side $2$ and diagonals both greater than $\sqrt8$, which is a geometric impossibility.)  We cut the quad cluster along the diagonal of length $\sqrt8$ and continue with deformations until the top edges on each simplex are $(y_4,y_5,y_6)=(2,2,\sqrt8)$.  We may apply \calc{474496219} and \calc{8990938295} to the two separate simplices to obtain the result.}
Now we may assume that the heights of $v_1,v_3$ are at most $2.3$.  If either
height is at least $2.1$, the result follows from \calc{5127197465}, which gives
the bound of $0$ on each half.

Finally, we have the case where both heights are at most $2.1$.  We may apply
dimension-reduction techniques so that that the each of the two
remaining corners $v\ne v_1,v_3$
of the quad cluster either has height $2$ or has distance $2$ from $v_1$ or $v_3$.
We then reprove Lemma~16.8 without using the bound of $0$ and Lemma~16.9 for tetrahedra without the bound  $-1.04\,\pt$.  This appears in \calc{1551562505} and \calc{3013446042}.  

If the dimension reduction drops the cross-diagonal $\{v_2,v_4\}$ all the way to $\sqrt8$, then we may swap diagonals and continue, until both diagonals are exactly $\sqrt8$.  In this case, by the cases already considered, we may assume that each corner has height at most $2.1$.  Also, geometric considerations give that the other edges are at most $2.02$:
  $$
  {\mathcal E}(2,2,2,\sqrt8,2.1,2.1,2,2,2.02) >\sqrt8.
  $$
The result follows in this case by \calc{4723770703}.

[p.241]  {\rmx {\it Mixed} is defined so as to include
the pure analytic case.  In earlier articles,
`mixed' excludes the pure analytic.  }
	{\it mixed\lto mixed or pure}
	
[p.243,\line+13,\line+14,\line+15]
	{\rmx Delete three sentences:}
	{\it `Let $v_{12}$ be $\ldots$  We let $\ldots$
	 Break the pentagon $\ldots$'}
	
[p.248,last displayed formula]  
	$=$ \lto $+$
{\rmx so that it reads}
	$$
	\sum_i f_{R_i}(D) \le \hat\sigma(Q_i) +
	\op{vor}_{R',0}(D) + \pi_R
	$$

[p.252,\S25.7,Cases~2 and 3]  {\rmx `The flat quarter'
is mentioned, but there are no flat quarters
that have been introduced into the context.  
This passage
has been displaced from its original context.}

[p.254,\line+7]
{\it to branch combine \lto to combine}
\end{\sz}

\bibliographystyle{abbrv}
\bibliography{all}

\bigskip

\begin{\sz}
{Research supported by NSF grant 0804189.}
\end{\sz}

\end{document}